\documentclass[review]{elsarticle}

\usepackage{hyperref}
\usepackage{lineno,hyperref}
\modulolinenumbers[100]

\journal{HAL/arXiv}

\usepackage{amsmath,amsfonts,amssymb,amsthm}
\usepackage{natbib}
\biboptions{authoryear}

\usepackage{graphicx,float}
\usepackage{color}

\usepackage[a4paper]{geometry}

\graphicspath{ {Fig/} }
\newcommand{\Xst}{\textbf{X}_{st}}
\newcommand{\Xt}{\textbf{X}_{t}}
\newcommand{\Xs}{\textbf{X}_{s}}
\newcommand{\xst}{\textbf{x}_{st}}

\newcommand{\xs}{\textbf{x}_{s}}
\newcommand{\E}{\mathrm{E}}
\newcommand{\dd}{\mathrm{d}}
\newcommand{\R}{\mathbb{R}}
\newcommand{\ass}{{[\mathcal{H}_2]}}
\newcommand{\assp}[1]{{[\mathcal{H}^\prime_{#1}]}}\newcommand{\inh}{\mathrm{inh}}
\newcommand{\B}{\mathcal B_{r,\tau}}

\newtheorem{proposition}{Proposition}

\graphicspath{{plots/}}
\begin{document}

\begin{frontmatter}

\title{Exploring first and second-order spatio-temporal structures of  lightning strike impacts in the French Alps  using subsampling}

\author[1]{Juliette Blanchet}
\ead{juliette.blanchet@univ-grenoble-alpes.fr}
\author[2]{Jean--François Coeurjolly \fnref{myfootnote}}
\ead{jean-francois.coeurjolly@univ-grenoble-alpes.fr}
\fntext[myfootnote]{Corresponding author}
\author[2]{Alexis Pellerin}
\ead{alexis.pellerin@univ-grenoble-alpes.fr}

\address[1]{Univ. Grenoble Alpes, CNRS, LJK, 38000 Grenoble, France}
\address[2]{Univ. Grenoble Alpes, CNRS, IRD, G-INP, IGE, F-38000 Grenoble, France}

\begin{abstract} 
We model cloud-to-ground lightning strike impacts in the French Alps over the period 2011-2021 (approximately 1.4 million of events) using spatio-temporal point processes. We investigate first and higher-order structure for this point pattern and address the questions of homogeneity of the intensity function, first-order separability and dependence between events. The tuning of nonparametric methods  and the different tests we consider in this study make the computational cost very expensive. We therefore suggest different subsampling strategies to achieve these tasks. 
\end{abstract}


\end{frontmatter}

\linenumbers


\bigskip\bigskip


\section{Introduction}

Lightning strike impacts, although continuously studied by scientists in the fields of physics, climatology or statistics remains a phenomenon with an important part of randomness. Recently, there has been notable advancements in the physical investigation of the phenomenon, specifically in understanding the formation of lightning strikes. The use of the LOFAR telescope (Low Frequency Array) has played a crucial role: it enables the precise recordings of locations and times~\citep{hare-etal-2018}. We focus here on the \emph{cloud-to-ground} category of lightning strikes (by opposition of the \emph{intra-cloud} category when the phenomenon is concentrated only inside the lightning cloud). The goal of this paper is to analyze, from a statistical point of view, this phenomenon, thus the locations and times of impacts.

Lightning strike impacts are naturally modelled by spatio-temporal point processes, which are stochastic processes modelling events in interaction. Whether they are spatial \citep{diggle2013statistical,moller2003statistical,baddeley2015spatial}, temporal \citep{daley2003introduction,daley2008introduction} or spatio-temporal \citep{diggle2006spatio}, point processes have known major developments these last 30 years and are now used in a large variety of fields of applications, for instance to model a disease in epidemiology \citep{gabriel2013stpp}, the propagation of forest fires \citep{serra2014spatio,opitz2020point,raeisi2023spatio}, the distribution of crimes in cities \citep[e.g.][]{mateu2023bayesian}

The dataset of interest, provided by Météorage, gives the coordinates (longitude/latitude) and the times of lightning strike impacts (in seconds) from 2011 to 2021. We  focus on events occurring mainly over the French Alps which includes a part of the Italian Alps and Piemont and the Mediterranean coast, see Figure~\ref{fig:zone_etude}. The spatial observation domain corresponds to $[4.43\mbox{°E}, 7.81\mbox{°E}] \times [43.1\mbox{°N}, 46.36\mbox{°N}]$. As a reminder, at the equator one degree is approximately equal to 110 kilometers. We adopt a North/South division of this domain following~\cite{auer2007histalp}. This division attempts to take into account geographical and climatological criteria. Such a division is standard in Alpine meteorological studies. It delimits the French Alps in two parts: the northern French Alps where perturbations are mainly generated by westerly flows coming from the Atlantic Ocean - we say they are under the Altlantic influence - and the southern French Alps where perturbation are mainly generated by southerly flows coming from the Mediterranean Sea - they are under the Mediterranean influence. Several studies linked extreme precipitation to the generating atmospheric influences and thus to Auer's climatological borders \citep{blanchet-etal-2021} but, to the best of our knowledge, it has never been done for lightning strikes. Thus it is pertinent to see whether differences in the distribution of lightning strikes occur between Alpine regions subject to different atmospheric influences. This spatio-temporal dataset contains approximately 1.4 million of events. This dataset is illustrated in Figure~\ref{fig:zone_etude}.

\begin{figure}[H]
    \centering
    \includegraphics[width=1\textwidth]{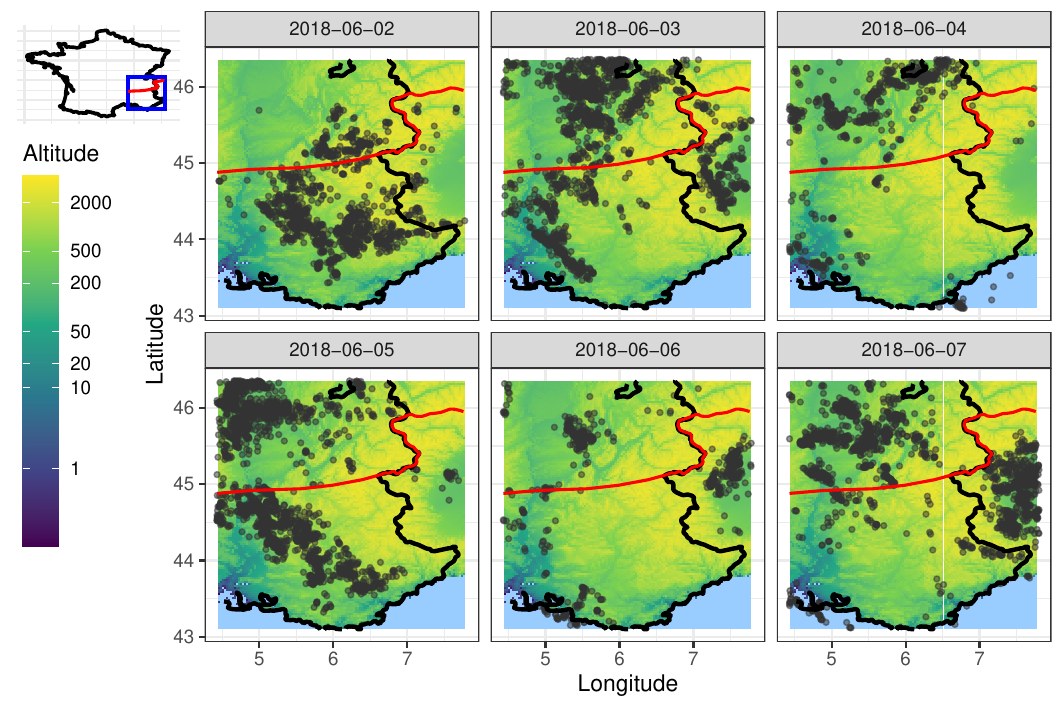}
    \caption{Locations of lightning strike impacts aggregated per day from June 2nd to June 7th 2018 (a quite active week) observed in the French Alps study domain, which includes a part of the Mediterrannean and Italian Alps. The red curve represents the Auer's climatological border between regions called 'North Alps' and 'South Alps'. Locations of impacts are superimposed on the altitude map for a better visualization.}
    \label{fig:zone_etude}
\end{figure}

Times, locations as well as the number of total events of lightning strikes are random and are therefore modelled by a spatio-temporal point process. In this paper, we intend to understand on the one hand the inhomogeneity of events across time and space and on the other hand dependence between these events. Usually, the first question is addressed by analyzing the (first-order) temporal, spatial or spatio-temporal intensity function and it is pertinent to investigate whether this intensity is separable in space and time \citep[see e.g.][]{moller2012aspects}. The second question can be tackled using higher-order summary statistics like the Ripley's $K$ function or the $J$ function \citep[see e.g.][]{diggle2013statistical,cronie2015aj} that measure departures from the homogeneous or inhomogeneous Poisson point process, the reference process modeling independent events in time and space. All these methodologies require fine tuning of hyperparameters (like bandwidth for kernel type estimators) or quite a large number of simulations for tests based on Monte-Carlo replications and global envelopes \citep[e.g.][]{myllymaki2017global,ghorbani2021testing}. Due to the considerable volume of events and with the aim of minimizing computational time, this paper focuses on exploring these two questions through the use of subsampling, a technique that has recently regained significance in spatial point pattern analysis \citep{cronie2023statistical}.

In the rest of this paper, we first present in Section~\ref{sec:background} a general background and notation on spatio-temporal point processes. Moments and summary characteristics for spatio-temporal point processes and versions of these characteristics after subsampling are presented. Section~\ref{sec:results} presents results on nonparametric intensity estimation, first-order separability tests and estimation of space-time subsampled Ripley's $K$ functions. Finally, a conclusion is presented in Section~\ref{sec:conclusion} while~\ref{app:simulation} illustrates an idea presented in Section~\ref{sec:subsampling}, which consists in subsampling an inhomogeneous point process to make it homogeneous.

\section{Spatio-temporal point processes and subsampling} \label{sec:background}

\subsection{Background on point processes}

We view a point process on a (complete metric) space set, say $S$, as a locally finite configuration of events in $S$ (see e.g.~\cite{moller2003statistical,daley2003introduction}). When $S=\R^+$ (resp. $S=\R^d$, $S=\R^d\times\R^+$) the stochastic models are often referred as temporal (resp. spatial, spatio-temporal) point processes. We focus in this paper on planar spatio-temporal point processes  $\Xst$ defined on $S= \mathbb{R}^2 \times \mathbb{R^+}$. Our observation domain is denoted by  $W \times T$ where $W \subset \mathbb{R}^2$ is a bounded spatial domain (here the French Alps) and $T \subset \mathbb{R^+}$ the period of observation (here 2011-2021).

Therefore, the data consists in a set
$\xst$ = $\left\{y_i=(x_i,t_i), i=1, \dots, n\right\}$ where  $x_i \in W$ and $t_i \in T$ respectively correspond to the observed location and time of the $i$th lightning strike. Thus $y_i$ stands for the space-time location. Without loss of generality, we order events by their observed times, $t_1 < ...< t_n$. Locations, times and $n$ are realizations of randoms variables and we assume that $\Xst$ is simple, meaning that $(x_i,t_i) \neq (x_j,t_j)$ for any $i\neq j$.

We let  $\Xt$ and $\Xs$ denote the point process aggregated in time or space, that is
\begin{equation*}
\Xt = \left\{t : y=(x,t) \in \Xst, x \in W\right\}
\quad \text{ and } \quad
\Xs = \left\{x : y=(x,t) \in \Xst, t \in T\right\}.
\end{equation*}
We let $N(A \times B) = \# \left\{(x,t) \in \Xst \cap  (A \times B) \right\}$ denote the counting variables, that is number of events in $A \times B$ where $A  \times B \subset \mathbb{R}^2 \times \mathbb{R^+}$. With a slight abuse of notation, we also denote $N(A)$ and $N(B)$ by $N(A) = \#\left\{x \in \Xs \cap  A \right\}$ and $N(B) = \#\left\{t \in \Xt \cap  B \right\}$. A point process is stationary (resp. isotropic) if its distribution is invariant under translations (resp. rotations). 
The distribution of $\Xst$ (as well as those of $\Xt$, $\Xs$) is characterized by the void probabilities, the finite dimensional distributions of counting variables or via their generating moment functionals (see e.g.~\cite{moller2003statistical}). These are usually difficult to establish. Intensity functions and summary statistics are simpler to define and estimate. Some of them are defined in the next two sections. 

\subsection{First-order intensity functions}

Assume that the processes $\Xst$, $\Xs$ and $\Xt$ have first-order intensities denoted by $\lambda_{st}$, $\lambda_{s}$ and $\lambda_{t}$. These functions can be interpreted for any $y=(x,t)$ with $x\in \R^2$ and $t\in \R^+$ by
\begin{equation}
    \lambda_{st}(y) = \lim_{|\dd y| \to 0}
    \frac{\E\{N(\dd y)\}}{|\dd y|};
    \quad 
    \lambda_{s}(x) = \lim\limits_{{\substack{|\dd x| \to 0}}} \frac{\E\{N(\dd x)\}}{|\dd x|}; 
    \quad
    \lambda_{t}(t) = \lim\limits_{{\substack{|\dd t| \to 0}}} \frac{\E\{N(\dd t)]}{|\dd t|}
    \label{eq:intensity}
\end{equation}
where $\dd y = \dd x \times \dd t$. By application of the Campbell theorem (see \citet{daley2003introduction}) these intensities are linked by 
\begin{equation}
    \lambda_s(x) = \int_T \lambda_{st}(x,t) \dd t
    \qquad \text{ and } \qquad
    \lambda_t(t) = \int_W \lambda_{st}(x,t)\dd x
    \label{eq:fct_inten}
\end{equation}
and are obviously linked to counting variables by
\begin{equation}
    \E\{N(A \times B)\} = \int_{A \times B} \!\!\!\!\!\! 
    \lambda_{st}(y) \dd y, \;\;\qquad   \E\{N(A)\} = \int_{A}\!\!\! \lambda_s(x) \dd x, \;\; \qquad 
    \E\{N(B)\} = \int_{B}\!\!\! \lambda_t(t) \dd t
    \label{eq:intensity_mesure}
\end{equation}
for any bounded Borel set $A\times B \subset \R^2\times\R^+$. A spatio-temporal point process is said to be homogeneous if $\lambda_{st}(y) = \lambda$ for any $y \in \R^2\times \R^+$. Otherwise, the process is said to be inhomogeneous. In this case, the intensity could depend only on $x$, only on $t$ or both on $x$ and $t$. And in the latter several authors (see e.g. \citet{moller2012aspects}) investigate a separability hypothesis of the intensity function. A spatio-temporal point process is said to be first-order separable if $\lambda_{st}(x,t)$ can be factorized as
\begin{equation}
    \lambda_{st}(x,t) = \lambda_{st}^{sep}(x,t)= \lambda_1(x) \lambda_2(t)  \quad \text{ for any } (x,t) \in W \times T
    \label{eq:sep}
\end{equation}
where $\lambda_1$ and $\lambda_2$ are non-negative and measurable functions respectively on $\R^2$ and $\R^+$. Under the hypothesis of first-order separability, i.e. if~\eqref{eq:sep} holds, the intensity of $\Xs$ and $\Xt$ can respectively be obtained as
\begin{equation}
    \lambda_s(x) = \lambda_1(x) \int_T \lambda_2(t) \dd t
    \quad \text{ and } \quad 
    \lambda_t(t) = \lambda_2(t) \int_W \lambda_1(x) \dd x.
    \label{eq:fct_inten_sep}
\end{equation}
Thus, by combining (\ref{eq:sep}) and (\ref{eq:fct_inten_sep}) under the first-order separability assumption the spatio-temporal intensity writes
\begin{align*}
    \lambda_{st}^{sep}(x,t) &= \frac{\lambda_s(x) \lambda_t(t)}{\nu}  \quad \text{ with } \nu=\E\{N(W\times T\}=\int_{W\times T}\lambda_{st}(x,t)\dd x\dd t. 
    \label{eq:fct_inten_seperable}
\end{align*}
Following \citet{ghorbani2021testing}, we introduce the natural summary functions $S_{st}(x,t)=\lambda_{st}(x,t)/\lambda_{st}^{sep}(x,t)= \nu \,\lambda_{st}(x,t)/ \{ \lambda_s(x)\lambda_t(t)\}$ and their spatial and temporal averages 
\begin{equation}
S_t(t) = |W|^{-1}\int_W S_{st}(x,t)\dd x 
\quad \text{ and } \quad 
S_x(x) = |T|^{-1}\int_T S_{st}(x,t) \dd t\label{eq:S}
\end{equation}
where $(x,t) \in W \times T$ and $\nu=\int_{W\times T}\lambda_{st}(x,t)\dd x \dd t$. Estimates of these quantities are easily obtained by plugging estimates of $\lambda_{st},\lambda_{s}$ and $\lambda_{t}$. Finally, from Campbell theorem $\nu=\E[N(W\times T)]$ can be estimated by the observed number of lightning strikes $n$. Under the first-order separability hypothesis, $S_{st}(x,t) = S_{s}(x)=S_{t}(t)= 1$. Departures to 1 of estimates of $S_s$ and $S_t$ can therefore indicate which spatial areas and/or time intervals are responsible for non-separability. 

Estimates of $\lambda_t,\lambda_s,\lambda_{st},S_t,S_s$ and $S_{st}$ and tests of first-order separability using global envelopes \citep{myllymaki2017global} are investigated in Section~\ref{sec:results}.

\subsection{Higher-order intensity functions and summary statistics}

If the $k$th ($k\ge 1$) moment measure of $\Xst$ is absolutely continuous with respect to the Lebesgue measure of $\R^2\times\R^+$, the $k$th order intensity exists and is defined as
\begin{equation*}
\lambda_{st}^{(k)}(y_1,\dots,y_k) = \lim_{|\dd y_1|\to 0,\dots,|\dd y_k|\to 0} \frac{\E\{ N(\dd y_1)\dots N(\dd y_k)\}}{|\dd y_1|\dots |\dd y_k|}
\end{equation*}
for any pairwise distinct $y_1,\dots,y_k \in \R^2\times \R^+$. The Poisson point process usually serves as the reference process modelling events without dependence. Poisson point processes have many interesting properties. In particular, for any $k\ge 1$, $\lambda_{st}^{(k)}(y_1,\dots,y_k)=\lambda_{st}(y_1)\dots \lambda_{st}(y_k)$. The pair correlation is an index of departure to Poisson assumption and is defined for any distinct $y_1,y_2\in \R^2\times \R^+$ by (see e.g. \citet{moller2003statistical})
\begin{equation*}
g_{st}(y_1,y_2) = \frac{\lambda^{(2)}_{st}(y_1,y_2)}{\lambda_{st}(y_1)\lambda_{st}(y_2)}    
\end{equation*}
where we use the convention $a/0=0$ for any $a\ge 0$. If $\Xst$ is stationary (resp. isotropic) then $\lambda_{st}$ is constant over space and time and $g_{st}$ depends only on $y_2-y_1$ (resp. $\|y_2-y_1\|$). For inhomogeneous (thus non stationary) point processes and to be closer to the space-time nature of $\Xst$ several frameworks and assumptions exist. Among them, we consider~$\ass$ and~$\assp{k}$
\begin{itemize}
\item[$\ass$] $\lambda_{st}^{(2)}$ exists, $g_{st}$ is invariant under translations and there exists a function $\bar g$  such that $g_{st}(y_1,y_2)=g(y_2-y_1)= \bar g(\|x_2-x_1\|,|t_2-t_1|)$.
\item[$\assp{k}$] $\lambda_{st}^{(k)}(y_1,\dots,y_k)$ exists and satisfies for any pairwise distinct $y_1,\dots,y_k \in \R^
2\times \R^+$ and any $a\in \R^2\times \R^+$  
\begin{equation*}
\frac{\lambda_{st}^{(k)}(y_1,\dots,y_k)}{\lambda_{st}(y_1)\dots \lambda_{st}(y_k)} =
\frac{\lambda_{st}^{(k)}(y_1+a,\dots,y_k+a)}{\lambda_{st}(y_1+a)\dots \lambda_{st}(y_k+a)}.
\end{equation*}
\end{itemize}
$\ass$ assumes that the pair correlation depends only on the spatial and temporal distances between events. It is particularly well-suited for spatio-temporal point processes where time and space typically serve distinct roles and operate on different scales. Under $\ass$, \citet{gabriel2013stpp} introduced the space-time Ripley's $K$ function as a very natural extension of the stationary version by
\begin{equation}\label{eq:Kst}
K_{\inh,st} (r,\tau) = \int_{\B} \bar g(\|x\|,|t|)\dd x \dd t
\end{equation}
where $\B$ is the cylindrical ball centered at $0$ given by 
$\B=\{y=(x,t)\in \R^2\times \R^+: \|x\|\le r, |t|\le \tau \}$. Under the Poisson case,  $g_{st}=\bar g_{st}=1$ and $K_{\inh,st} (r,\tau)= |\B| = 2\tau \, \pi r^2$. 

Inhomogeneous versions of standard statistics such as the empty space function $F$, the nearest-neighbour distribution function $G$ and the $J$ function are less straightforward to derive. \citet{cronie2015aj} consider intensity reweighted moment stationary (IRMS) spatio-temporal point process, that is models satisfying $\assp{k}$ for any $k\ge 1$, and such that
\begin{equation}\label{eq:assCVl}
\bar \lambda >    0 
\quad \text{ and } \quad
\limsup_{k\to \infty} \left( \frac{\bar \lambda^k}{k!} I_k(r,\tau)\right)^{1/k} <1
\end{equation}
where  $\bar \lambda=\min_{y\in \R^2\times \R^+}\lambda_{st}(y)>0$ and where for any $k\ge 1$
\begin{align*}
I_k(r,\tau) &= \int_{\B} \dots \int_{\B} \frac{\lambda_{st}^{(k)}(y_1,\dots,y_k)}{\lambda_{st}(y_1)\dots \lambda_{st}(y_k)} \dd y_1 \dots \dd y_k \\
\tilde I_k(r,\tau) &= \int_{\B} \dots \int_{\B} \frac{\lambda_{st}^{(k+1)}(0,y_2,\dots,y_{k+1})}{\lambda_{st}(0)\dots \lambda_{st}(y_{k+1})} \dd y_2 \dots \dd y_{k+1}. 
\end{align*}
Under this framework, \citet{cronie2015aj} use series expansions of the moment generating functionals to propose the following extensions of $F,G,J$, provided here with little details, for any $r,\tau>0$ 
\begin{align*}
1-F_{\inh,st}(r,\tau)&= 1 + \sum_{k \ge 1} \frac{(-\bar \lambda)^k}{k!}    I_k(r,\tau) \\
1-G_{\inh,st}(r,\tau)&= 1 + \sum_{k \ge 1} \frac{(-\bar \lambda)^k}{k!}    \tilde I_k(r,\tau) \\
J_{\inh,st} &= \frac{1-G_{\inh,st}(r,\tau)}{1-F_{\inh,st}(r,\tau)}.
\end{align*}
To see the intuition, it can be shown (see again~\citet{cronie2015aj}) that for stationary spatio-temporal point process, $1-F_{\inh,st}(r,\tau)=\mathrm P(N(\B)>0)$ and $1-G_{\inh,st}=\mathrm P(N(\B)>0 \mid 0 \in \Xst)$ which correspond to standard definitions of these functions. Finally, under the Poisson assumption, we may check that $1-F_{\inh,st}(r,\tau)=1-G_{\inh,st}(r,\tau) =\exp(-\bar \lambda|\B|)$ and $J_{\inh,st}(r,\tau)=1$. In the rest of the paper, we assume that $\Xst$ satisfies $\ass$, $\assp{k}$ for any $k\ge 1$ and~\eqref{eq:assCVl}.

\subsection{Subsampling point processes} \label{sec:subsampling}

The nonparametric estimation of intensity functions, tests of first-order separability, goodness-of-fit tests of inhomogeneous Poisson point processes require a tuning of few parameters (like bandwidths parameters) and/or simulations which can lead to severe computational cost. This mainly comes from the highly inhomogeneous spatio-temporal nature of the lightning strikes dataset (there are many areas and large periods of time where no lightning strike is observed) and the large number of observed events. To reduce these costs, it does therefore make sense to consider subsamples of the dataset.

Subsampling  point processes has regained popularity recently (see e.g. \citet{chiu2013stochastic,cronie2023statistical}) in particular in the context of statistical learning, cross-validation technique, variance estimation, etc. Independent subsampling corresponds to the process of thinning (or conversely retaining) a point from an initial point pattern. Let $\pi:\R^2\times\R^+ \to [0,1]$ and $(\varepsilon(y), y\in \R^2\times \R^+)$ a random field of independent Bernoulli distributions with parameter $\pi(y)$, we define the subsampled version $\Xst^\pi$ as the thinning with probability $1-\pi$ (or conversely with retaining probability $\pi$) as
\begin{equation*}
  \Xst^\pi = \left\{\;  y \in \Xst : \varepsilon(y) =1 \; \right\}.  
\end{equation*}
When $\pi(\cdot)=\pi_0$ is constant over space and time, $\Xst^{\pi_0}$ is a space-time independent subsampling of $\Xst$. In this section we briefly review properties of $\Xst^\pi$. When we add the superscript $\pi$ to a characteristic, for instance $\lambda^{(k),\pi}_{st}$, $g_{st}^{\pi}$ we mean the $k$th order intensity function, the pair correlation function of $\Xst^\pi$. 
Iterated versions of Campbell theorem allow to prove, see e.g.~\cite{cronie2023statistical}, that 
\begin{equation}
\lambda^{(k),\pi} (y_1,\dots,y_k) = \pi(y_1)\dots \pi(y_k)\lambda^{(k)}(y_1,\dots,y_k) = \pi_0^k \lambda^{(k)}(y_1,\dots,y_k) \label{eq:lkpi}
\end{equation}
where the latter holds if $\pi(\cdot)=\pi_0$. From~\eqref{eq:lkpi}, we deduce the following proposition.

\begin{proposition} \label{prop:Xstpi} Under the previous assumptions on $\Xst$ and general notation, we have the following statements
\begin{itemize}
\item[(i)] $\Xst^\pi$ necessarily satisfies the same assumptions as $\Xst$, namely $\ass$, $\assp{k}$ for any $k$ and~\eqref{eq:assCVl} if $\bar \pi=\inf_y \pi(y)>0$.
\item[(ii)] If $\pi(y)=\pi_0$, 
\begin{equation*}
S_{st}^{\pi_0}(x,t)=S_{st}(x,t), \quad S_{s}^{\pi_0}(x)=S_{s}(x) \quad \text{ and } \quad S_{t}^{\pi_0}(t)=S_{t}(t). 
\end{equation*}
\item[(iii)] The pair correlation and the space-time Ripley's $K$ function satisfy for any $y_1\neq y_2$ and any $r,\tau>0$
\begin{equation*}
g_{st}^\pi(y_1,y_2) = g_{st}(y_1,y_2) 
\quad \text{ and } \quad    
K_{\inh,st}^\pi(r,\tau) = K_{\inh,st}(r,\tau).
\end{equation*}
\item[(iv)] From (i), $\Xst^\pi$ is an IRMS, for any $r,\tau>0$, $I_k^\pi(r,\tau)=I_k(r,\tau)$, $\tilde I_k^\pi(r,\tau)=\tilde I_k(r,\tau)$ which yields
\begin{align}
1-F^\pi_{\inh,st}(r,\tau)&= 1 + \sum_{k \ge 1} \frac{(-\bar \lambda \bar \pi)^k}{k!}    I_k(r,\tau) \label{eq:Finhpi}\\
1-G^\pi_{\inh,st}(r,\tau)&= 1 + \sum_{k \ge 1} \frac{(-\bar \lambda \bar \pi)^k}{k!}    \tilde I_k(r,\tau).\label{eq:Ginhpi}
\end{align}
\end{itemize}
\end{proposition}

Proof of Proposition~\ref{prop:Xstpi} is quite straightforward and essentially follows from~\eqref{eq:lkpi}. Note that (ii) is the only result for which  the subsampling is homogeneous. This result would not be true even if the retaining probability field is separable in space and time, that is if $\pi(y)=\pi(x,t)=\pi_s(x)\pi_t(t)$. Results~\eqref{eq:lkpi}, (i)-(iii) show that we can recover characteristics  like intensity functions of $\Xst$ from the ones of $\Xst^\pi$. So to estimate $g_{st}$ or $K_{\inh,st}$, we simply estimate $g_{st}^\pi$ and $K_{\inh,st}^\pi$. To estimate $\lambda_{st}(y)$, we can estimate $\lambda_{st}^\pi$ and set $\hat\lambda_{st}(y)=\hat\lambda_{st}^\pi(y)/\pi(y)$. Result (iv) was proved by~\cite{cronie2015aj}. It tells that the subsampling has a more complex effect on distance-based summary statistics like $F_{\inh,st}, G_{\inh,st}$ and $J_{\inh,st}$. This remark would still apply even if $\Xst$ were stationary and/or if we consider a homogeneous subsampling.

Going back to the general objective of this section which is to reduce computational cost of nonparametric estimation, we propose two subsampling strategies:
\begin{itemize}
\item[(a)] $\pi(y)=\pi_0$ with $\pi_0\in(0,1)$. In particular, we view $\pi_0$ as a small positive real number.
\item[(b)] Let $\mu>0$ and $\pi(y)=\frac{\mu}{\lambda_{st}(y)}\mathbf 1(y\in \Delta_\mu)$ where $\Delta_\mu=\{y\in \R^2\times \R^+ : \lambda_{st}(y)\ge \mu\}$.
\end{itemize}
Note that if one sets the expected number of points $\nu^\pi$  of $\Xst^\pi$, we can set $\pi_0$ and $\mu$ as follows
\begin{equation*}
\nu^\pi = \E \{N^\pi(W\times T)\} = \int_{W\times T} \pi(y)\lambda_{st}(y)\dd y = \left\{
\begin{array}{cc}
    \pi_0 \E\{N(W\times T)\} &\text{(a)}  \\
     \mu |\Delta_\mu|&\text{(b)} 
\end{array}
\right.
\end{equation*}
For (a), $ \E\{N(W\times T)\}$ is estimated by $n$ the observed number of data points and so $\pi_0=\nu^\pi/n$. For (b), the problem may not have a solution (or a unique solution) but $\mu$ could be obtained using an optimization procedure (see~\ref{app:simulation} for more details).

Strategy (a) is very natural. We use no information from $\Xst$ and subsample independently of space and time. Strategy (b) is more tricky. Indeed, in particular
\begin{equation*}
\lambda_{st}^\pi(y)  = \pi(y) \lambda_{st}(y) = \mu \mathbf{1}(y\in \Delta_\mu).
\end{equation*}
So ${\Xst^\pi} \Large|_{\Delta_\mu}$ becomes a homogeneous point process with intensity $\mu$ on $\Delta_\mu$. This strategy suffers from the obvious drawback that $\lambda_{st}$ is unknown and we believe this strategy is uninteresting to estimate $\lambda_{st}$ or to test the first-order separability. However, for the estimation of $K_{\inh,st}(r,\tau)$, the interest is clear: $\Xst^\pi$ is homogeneous on $\Delta_\mu$ and has the same second-order structure as $\Xst$. So, for instance, to construct global envelopes of $K_{\inh,st}^\pi$ under the Poisson assumption, we simply have to simulate homogeneous Poisson point processes on $\Delta_\mu$. \ref{app:simulation} explores these ideas: we propose a practical procedure and simulation study (only for planar spatial point processes). We show that strategy (b) has some strong merit and deserves to be investigated in a future research. However, we did not pursue this idea for the lightning strikes dataset. The main difficulties are that the intensity $\lambda_{st}$ is highly inhomogeneous (lots of peaks and empty 'space-time' areas). This makes the set $\Delta_\mu$ really non convex and thus the estimation of the space-time (homogeneous) Ripley's $K$ function really complex and unstable as it must take into account border correction.

We end this section with the following proposition which shows the impact of strategy (a), with a small value of $\pi_0$, for the estimation of distance-based summary statistics. This result provides a more formal approximation of \citet[bottom of p.16]{cronie2015aj}.

\begin{proposition} \label{prop:FGJ}
Let $\Xst^{\pi_0}$ be a subsampled version of $\Xst$ with $\pi(\cdot)=\pi_0>0$. We assume~$\ass$, $\assp{k}$ for any $k\ge 1$ and~\eqref{eq:assCVl}. Then under the notation of Proposition~\ref{prop:Xstpi}, for any $r,\tau>0$, we have the following statements as $\pi_0\to 0$
\begin{itemize}
\item[(i)] 
\begin{equation}\label{eq:Finhpi0}
1 - F_{\inh,st}^{\pi_0} (r,\tau) = 
\left\{ 
1 - F_{\inh,st}^{\pi_0,\mathrm{Poisson}} (r,\tau) 
\right\} 
\;
\left\{
1 + \pi_0^2 \frac{\bar\lambda^2}2 I_2(r,\tau)  + o(\pi_0^2)
\right\}
\end{equation}
\item[(ii)]
\begin{equation}\label{eq:Ginhpi0}
1 - G_{\inh,st}^{\pi_0} (r,\tau) = 
\left\{ 
1 - G_{\inh,st}^{\pi_0,\mathrm{Poisson}} (r,\tau) 
\right\} 
\;
\left\{
1 + \frac{\bar\lambda^2\pi_0^2}2  \tilde I_2(r,\tau)  + o(\pi_0^2)
\right\}
\end{equation}
\item[(iii)]
\begin{align}\label{eq:Jinhpi0}
J_{\inh,st}^{\pi_0} (r,\tau) &= 
1 - \bar \lambda \pi_0 \left\{ K_{\inh,st}(r,\tau)-|\B|\right\} 
\nonumber\\
&+ \;
\frac{\bar\lambda^2}2 \pi_0^2\left\{
(K_{\inh,st}(r,\tau)-|\B|)^2 + \tilde I_2(r,\tau)-I_2(r,\tau)
\right\} + o(\pi_0^2)
\end{align}
\end{itemize}
\end{proposition}

\begin{proof}
Let $(u_n)_n$ be a sequence of real numbers such that $u_n=o(1/n)$ as $n\to\infty$, we leave the reader to prove that for any $c>0$ 
\begin{equation}\label{eq:dl}
\left(1+\frac{c}n+u_n\right)^n = \exp(c)\left(1+\frac12nu_n^2 + o(1/n)\right).
\end{equation}
Now, consider~\eqref{eq:Finhpi}. Since from~\eqref{eq:assCVl}
$$
\sum_{k\geq 3} \frac{(-\bar\lambda \pi_0)^k}{k!} I_k(r,\tau) \le \pi_0^3 \sum_{k\ge } \frac{(-\bar\lambda \pi_0)^k}{k!} I_k(r,\tau)= O(\pi_0^3)
$$
we deduce that 
$$
1- F_{\inh,st}^{\pi_0}(r,\tau)=1 - \bar \lambda \pi_0 I_1(r,\tau) +  \frac{\bar\lambda^2}2 \pi_0^2 I_2(r,\tau) + o(\pi_0^2).
$$
Using~\eqref{eq:dl} we then obtain, by noticing that $I_1(r,\tau)=|\B|$
\begin{align*}
\left\{1- F_{\inh,st}^{\pi_0}(r,\tau) \right\}^{1/\pi_0} 
&=\exp\left( -\bar \lambda |\B| \right) 
\left( 
1+ \frac{\bar\lambda^2}2 \pi_0 I_2(r,\tau) +o(\pi_0)
\right)
\end{align*}
which leads to (i) by reminding of course that $1- F_{\inh,st}^{\pi_0,\mathrm{Poisson}}(r,\tau)= \exp(-\pi_0\bar\lambda |\B|)$. We now consider~\eqref{eq:Ginhpi}. Since~\eqref{eq:assCVl} also implies that
$$
\sum_{k\ge 3} \frac{(-\bar\lambda \pi_0)^k}{k!}\tilde I_k(r,\tau) =O(\pi_0^3)
$$
we have using~\eqref{eq:dl}
\begin{align*}
\left\{1- G_{\inh,st}^{\pi_0}(r,\tau) \right\}^{1/\pi_0} &=   
\exp\left( -\bar \lambda \tilde I_1(r,\tau)\right) 
\left( 
1 + \frac{\bar \lambda \pi_0^2}{2} \tilde I_2(r,\tau)+ o(\pi_0^2)
\right) \\
\left\{1- G_{\inh,st}^{\pi_0}(r,\tau) \right\}^{1/\pi_0}&= \left( 1- G_{\inh,st}^{\pi_0,\mathrm{Poisson}} \right)
\exp\left(
-\bar\lambda \pi_0(\tilde I_1(r,\tau)-I_1(r,\tau))
\right)  \\
&\qquad\qquad
\times\left( 
1 + \frac{\bar \lambda^2 \pi_0^2}{2} \tilde I_2(r,\tau)+ o(\pi_0^2)
\right).
\end{align*}
The result is obtained by noticing that $\tilde I_1(r,\tau)=\int_{\B} g_{st}(y)\dd y = K_{\inh,st}(r,\tau)$ and by applying a Taylor expansion of the exponential function. Finally,~\eqref{eq:Jinhpi0} follows easily from~\eqref{eq:Finhpi0}-\eqref{eq:Ginhpi0}.
\end{proof}

Thus, when $\pi_0$ is small, $1-F_{\inh,st}^{\pi_0}(r,\tau)= 1-F_{\inh,st}^{\pi_0}(r,\tau)(1+o(\pi_0))=(1-\exp(-\bar\lambda \pi_0 \B))(1+o(\pi_0))$ and $J_{\inh,st}^{\pi_0}=1 + O(\pi_0)$, so $\Xst^{\pi_0}$ is hard to distinguish from an inhomogeneous Poisson point process. One could be tempted to define $\check J_{\inh,st}^{\pi_0}(r,\tau)= (1-J_{\inh,st}^{\pi_0}(r,\tau))/\pi_0$ to discard the first-order term but in that case
\begin{equation*}
\check J_{\inh,st}^{\pi_0}(r,\tau) = \bar \lambda \left(  K_{\inh,st}^{\pi_0}(r,\tau) - |\B|\right) + o(\pi_0).
\end{equation*}
Hence, upto $o(\pi_0)$, $\check J_{\inh,st}^{\pi_0}(r,\tau)$ only measures the difference between the inhomogeneous $K$ function and the one under the inhomogeneous Poisson assumption. Proposition~\ref{prop:FGJ} combined with the fact that $\bar \lambda$ is extremely small on the lightning strikes dataset making the estimation of $J_{\inh,st}$ very unstable, we have decided to skip distance-based summary statistics in the data analysis.

To summarize this section, among different strategies for subsampling, the one assuming no information a priori seems the more appropriate to estimate characteristics of $\Xst$. We typically use $\pi_0\approx 2.5\%$, so that the number of points in $\Xst^{\pi_0}$ is close to 35,000 points, which is still quite large. According to Propositions~\ref{prop:Xstpi}-\ref{prop:FGJ},  Section~\ref{sec:results} presents empirical results for the estimation of $\lambda_{st},S_{st}, S_s, S_t$ and $K_{\inh,st}$. 


\section{Results on the dataset } \label{sec:results}

\subsection{Intensity estimation}

In this section, we present empirical results for the nonparametric estimations of $\lambda_t,\lambda_s$ and $\lambda_{st}$. We consider the kernel intensity estimator with edge effect correction (see e.g. \citet{baddeley2015spatial}). To estimate $\lambda_{st}$ instead of a 3-dimensional kernel, we follow \citet{ghorbani2021testing} and use a product of two kernels (one in space and one in time). For any $y=(x,t)\in W\times T$, we define for two bandwidths parameters $b_s,b_t$
\begin{equation*}
\hat \lambda_{st}(y;b_s,b_t) = \sum_{y^\prime=(x^\prime,t^\prime)\in \Xst} 
\frac{k_{s,b_s}(x^\prime-x)}{e_s(x^\prime;b_s)} \;
\frac{k_{t,b_t}(t^\prime-t)}{e_t(t^\prime;b_t)} 
\end{equation*}
where $k_{s,b_s}=b_s^{-2}k_s(\cdot/b_s)$, $k_{t,b_t}=b_t^{-1}k_t(\cdot/b_t)$ and where $k_s$ (resp. $k_t$) is a two-dimensional (resp. one-dimensional) kernel function. The terms $e_s$ and $e_t$ serve as edge correction factor and we use the Diggle's correction \citep{baddeley2015spatial} given by
$$
e_s(x^\prime;b_s)=\int_W k_{s,b_s}(x-x^\prime)\dd x
\quad \text{ and }\quad
e_t(t^\prime;b_t)=\int_T k_{t,b_t}(t-t^\prime)\dd t.
$$
Estimates of $\lambda_s$ and $\lambda_t$ can be obtained from $\hat \lambda_{st}$ or directly from $\Xs$ and $\Xt$. We use the latter and define for any $y=(x,t)\in W\times T$ and two bandwidths parameters $\tilde b_s$ and $\tilde b_t$
\begin{equation*}
\hat \lambda_{s}(x;\tilde b_s) = \sum_{x^\prime\in \Xs} 
\frac{k_{s,\tilde b_s}(x^\prime-x)}{e_s(x^\prime;\tilde b_s)} 
\quad \text{ and } \quad
\hat \lambda_{t}(t;\tilde b_t) = \sum_{t^\prime\in \Xt} 
\frac{k_{t,\tilde b_t}(t^\prime-t)}{e_t(t^\prime;\tilde b_t)}. 
\end{equation*}
As illustrated by \citet{bivand2008applied}, the choice of the kernel to be used has less impact than the smoothing parameters. We simply used isotropic  Gaussian kernels: thus $k_s$ and $k_t$ depend only on $||x^\prime-x||$ and $|t^\prime-t|$ respectively. We select the smoothing parameter $\tilde b_t$ as proposed by~\citet{sheather1991reliable}. Regarding the choice of $\tilde b_s$, we follow the simulation done by~\citet{cronie2018non} and use a similar approach. They suggest to select $\tilde b_s$ as the parameter achieving the smallest squared inverse residuals value (see \citet{baddeley2005residual}) given by
\begin{equation}\label{eq:loss}
\mathcal L(\tilde b_s) =
\left\{
\sum_{x^\prime \in \Xs} \frac{1}{\hat \lambda_s(x^\prime;\tilde b_s)} - |W|
\right\}^2.
\end{equation}
In this procedure, $\hat \lambda_s(\cdot;\tilde b_s)$ has to be estimated precisely at data points for any $\tilde b_s$. This can be extremely computationally expensive if one uses the estimator proposed by~\citet{baddeley2015spatial} for that task. Instead, we use a subsampling strategy combined with a 10-fold cross-validation. We first substitute $\Xs$ with a subsampled version $\Xs^\pi$ (obtained with $\pi(\cdot)=\pi_0=2.5\%$). Then, in~\eqref{eq:loss} we estimate $\lambda_s$ using 90\% of data from $\Xs^\pi$ and evaluate the loss function at the 10\% remaining points. We repeat the estimation of the loss function 10 times, then compute the average loss and finally the optimal $\tilde b_s$. We also repeat the subsampling procedure 50 times and average the $\tilde b_s$. 

Estimates of $\lambda_t$ and $\lambda_s$ are evaluated over a regular grid of respectively 1,000 times points between 01/01/2011 - 00:00:00 and 31/12/2021 - 23:59:59 and a $256 \times 256$ pixels grid. The estimation of $\lambda_{st}$ is more expensive and we use a homogeneous subsampling with $\pi_0=2.5\%$ (such that $\nu^\pi\approx 35,000$) and Proposition~\ref{prop:Xstpi}. Also, to reduce computational complexity, we use $b_s=\tilde b_s$ and $b_t=\tilde b_t$ for the estimation of $\lambda_{st}$. Figures~\ref{fig:intensity_temporal}-\ref{fig:intensity_spatial} depict empirical results for estimates of $\lambda_s$ and $\lambda_t$ only.

Let us comment Figure \ref{fig:intensity_temporal} which explores variations between seasons and between the northern and southern parts of the French Alps. The impact rate in the northern Alps seems to be, on average, more than twice as low as that in the southern Alps, across all years and seasons. This is probably related to the Mediterranean influence that is predominant in the southern region~\citep{blanchet-etal-2021}, a region known to generate more convective events, which are associated to electrical activity. Furthermore, there is a significant year-to-year variability in the estimated intensity. Specifically, 2017 appears to be an exceptional year characterized by a low number of lightning strikes throughout the year, with an average intensity about four times lower than that observed in the three areas. The year 2017 was indeed characterised by predominant anticyclonic conditions responsible for a significant shortfall in rainfall over almost the entire country. Conversely, 2021 recorded nearly three times the average number of impacts in the three areas. In 2021, lightning strikes were more frequent in the northern regions during spring and in the southern regions during autumn. Considering that these regions experience distinct atmospheric influences (Atlantic vs. Mediterranean), this highlights the role of atmospheric circulation in the spatiotemporal variability of lightning strikes. Overall, the impact intensity is lower in winter and higher in summer (with summer intensity nearly 40 times the average winter intensity), attributed to the predominance of convective events during the summer season.

We now comment Figure \ref{fig:intensity_spatial} which displays estimates of $\widehat{\lambda}_s(x)$ over a $256 \times 256$ pixel grid. We estimate $\lambda_s$ for the whole period as well as every year from 2011 to 2021. The Italian region seems to be prone to a significantly higher number of lightning strike impacts, with an intensity exceeding ten times the regional average.The Piemont region is indeed recognized for experiencing "East returns," which refer to perturbed weather patterns leading to intense storms and the potential for significant precipitation, as highlighted in the work of~\citep{blanchet-etal-2021b}. In general, Auer's border effectively delineates the spatial intensity of lightning strikes, revealing higher intensities in the southern part of the Alps compared to the northern part. This is in agreement with the conclusions obtained from~Figure \ref{fig:intensity_temporal}. Moreover, the annual maps indicate significant year-to-year variability, particularly in 2017, which stands out as a markedly different year with an intensity nearly six times lower than the overall country's average, as depicted in Figure \ref{fig:intensity_temporal}.

\begin{figure}[htbp]
\centering
\includegraphics[width=\textwidth]{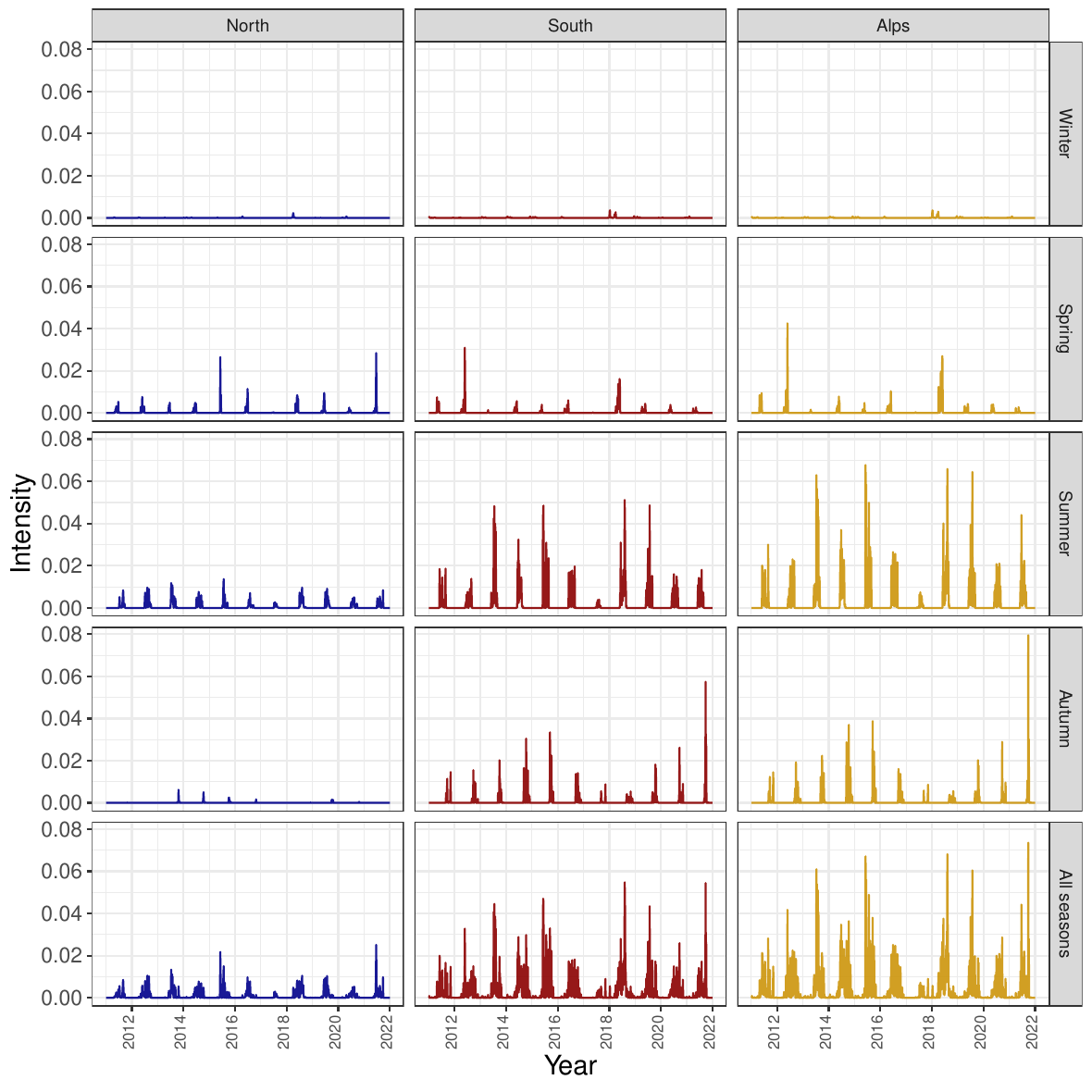}
\caption{Nonparametric temporal intensity estimated by study area (Northern Alps, Southern Alps, All the French Alps) for the 2011-2021 period; 
Nonparametric temporal intensity estimated by study area (Northern French Alps, Southern French Alps, All the French Alps) by season (Winter=December to March, Spring=April to May, Summer=June to August, Autumn=September to November).}
\label{fig:intensity_temporal}
\end{figure}

\begin{figure}[htbp]
\centering
\includegraphics[width=1\textwidth]{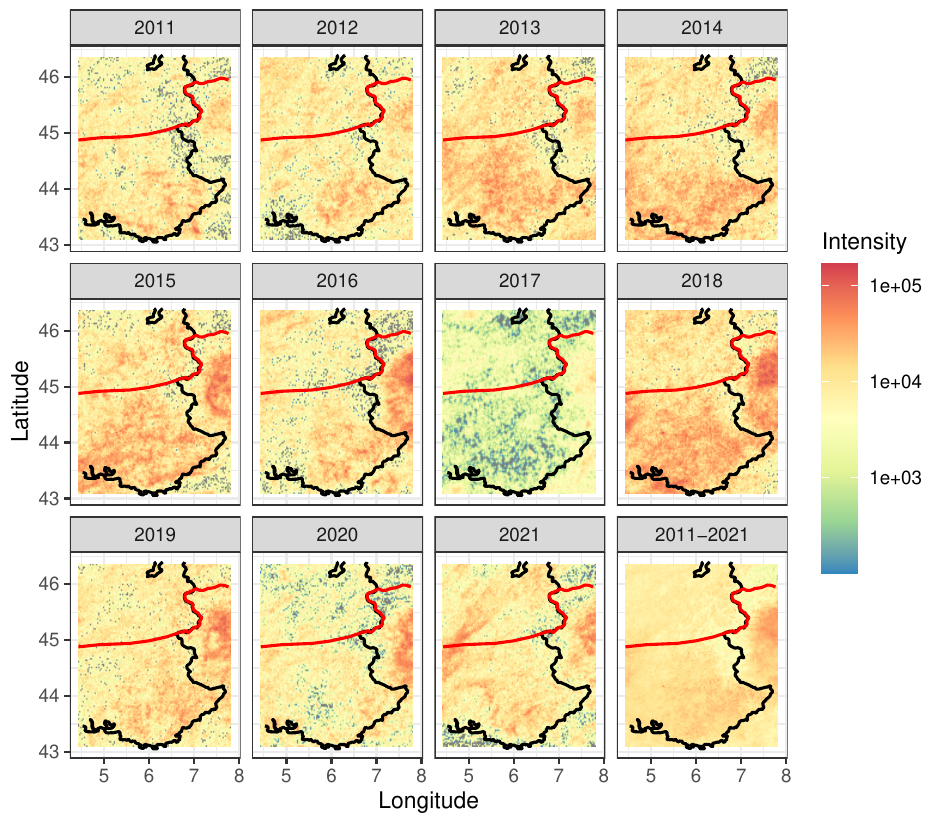}
\caption{Nonparametric spatial estimates of $\lambda_{s}$ aggregated per year from 2011 to 2021 and over the whole period 2011-2021 (thus a nonparametric estimate of $\lambda_s$). The red line for both figures marks the North/South division of the Alps according to \cite{auer2007histalp}. For a better visualization, values of intensity smaller than 100 were discarded (less than 0.7\% of values) and the values for the period 2011-2021 have been divided by 11 for a better comparison.
}
\label{fig:intensity_spatial}
\end{figure}

\subsection{First-order separability testing}

Based on the previous section, we now investigate the first-order separability hypothesis of the intensity function, that is the null hypothesis  given by~\eqref{eq:sep}. We use the summary statistics $S_{st}, S_s$ and $S_t$ presented in Section~\ref{sec:background}, see~\eqref{eq:S}. Since we estimate $\lambda_{st}$ using subsampling, we proceed similarly for these summary statistics, as guaranteed by Proposition~\ref{prop:Xstpi}. Then, we follow the statistical procedure proposed by~\citet{ghorbani2021testing} based on permutations to test the null hypothesis and in particular to have replications of $S_{st}, S_t$ and $S_s$ under the null. Finally, we use global envelope tests for $S_t$ and $S_s$ \citep{myllymaki2017global} to have a Monte-Carlo valid p-value and graphical interpretation of departures to the null. We have reproduced the whole of this procedure on two versions of $\Xst^{\pi_0}$ to investigate the robustness of findings. Figure~\ref{fig:separability_result} presents empirical results. Even if there are some (expected) slight differences between the two versions of $\Xst^{\pi_0}$, the conclusions are clear. In the two versions of $\Xst^{\pi_0}$, there is a strong evidence that the spatio-temporal intensity function is not separable in space and time. It can be also be pointed out that the first-order separability test is rejected mainly during summers periods mainly either in the southern Alps or in the part of the northern French Alps that experiences East return events (the Savoie and Haute-Savoie departments close to the Italian border). This may be attributed in both cases to the occurrence of convective storms showing strong electrical activity over short time periods and limited areas, inducing strong space-time dependencies.

\begin{figure}[htbp]
\centering
\includegraphics[width=0.4\textwidth]{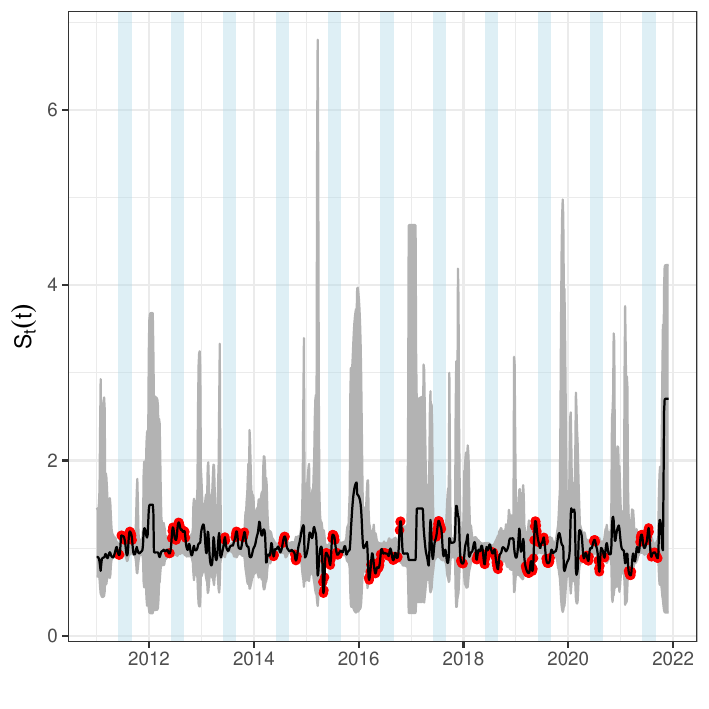}
\includegraphics[width=0.46\textwidth]{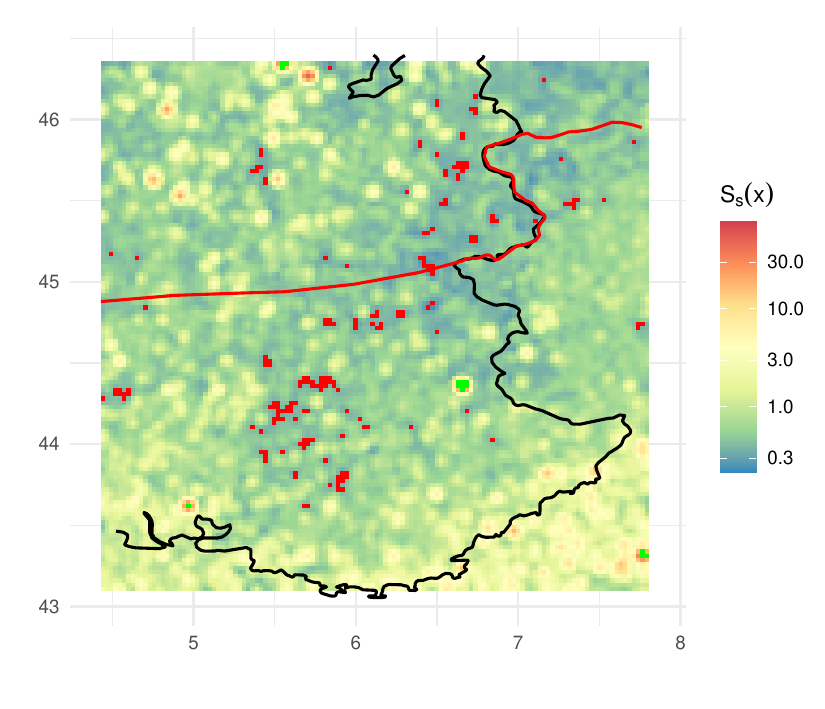}
\includegraphics[width=0.38\textwidth]{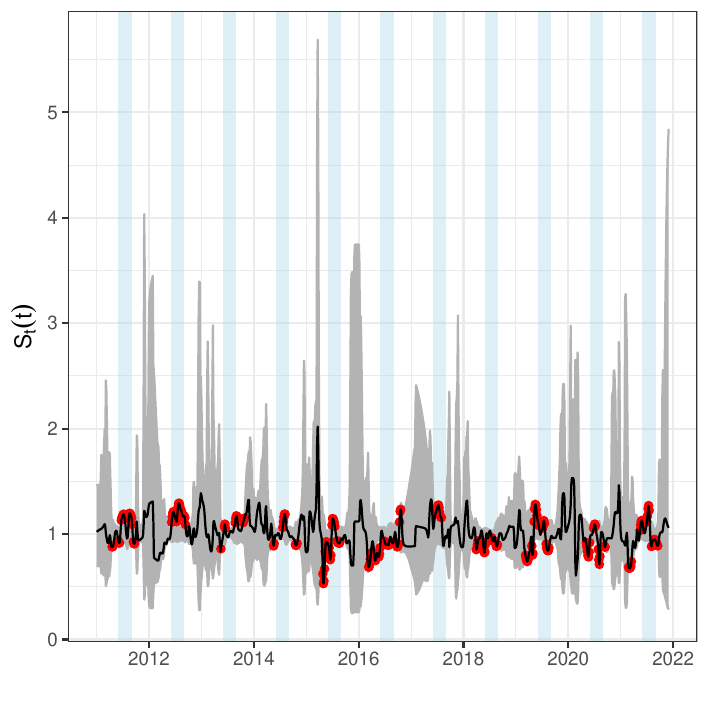}
\includegraphics[width=0.46\textwidth]{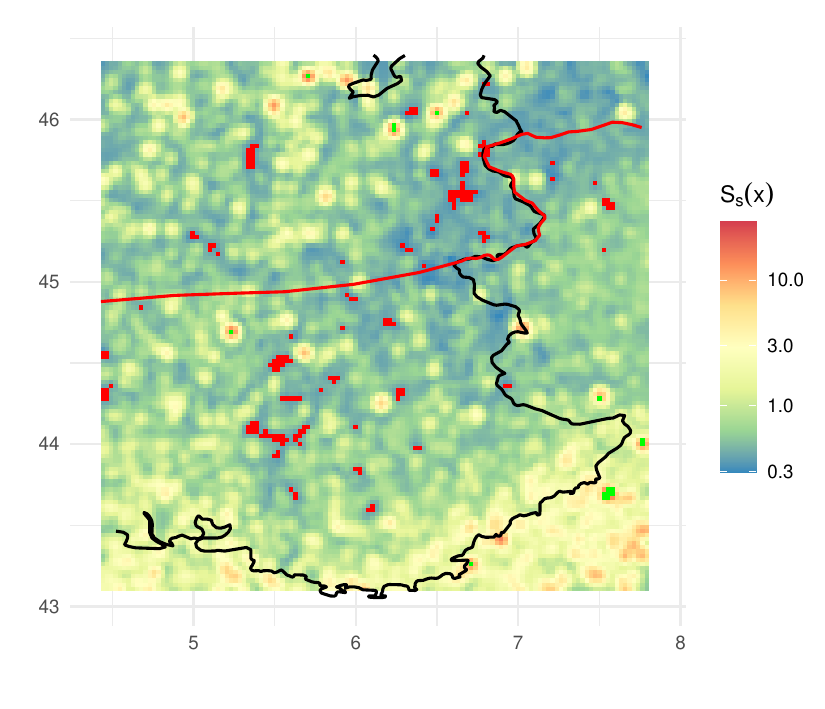}
\caption{Results for the first-order separability  testing procedure obtained from two subsampled versions of $\Xst^{\pi_0}$ (first and second rows). 95\% global envelopes tests using the ERL procedure were used where permutations are used to generate data under the null. The considered summary statistics are $S_t(\cdot)$ (left, solid curve) and $S_s(\cdot)$ (right, latent image). Envelopes are depicted only for $S_t$ and red dots indicate times or pixels for which the estimated summary statistics is outside the envelopes. Light blue rectangles corresponds to summer seasons (from June to end of September).
The adjusted global p-value for each of these envelope test is $p=0.05\%$.}
\label{fig:separability_result}
\end{figure}

\subsection{Space-time Ripley's $K$ function} \label{sec:K}

We now dig in the higher-order structure of $\Xst$. As specified at several places, we consider only subsampled versions of $\Xst^{\pi_0}$ with $\pi_0=2.5\%$ (so that $\nu^\pi\approx 35,000$). As a consequence  of Section~\ref{sec:subsampling} and in particular in light of Propositions~\ref{prop:Xstpi}-\ref{prop:FGJ}, we consider only the space-time inhomogeneous Ripley's $K$ as relevant summary statistic, that is  $K_{\inh,st}(r,\tau)=K_{\inh,st}^{\pi_0}(r,\tau)$ given by~\eqref{eq:Kst}.
Let $r \in [r_m,r_M]$ and $\tau\in [\tau_m,\tau_M]$.
To derive a simple graphical interpretation, we also suggest to consider and estimate the following one-dimensional functions $ K_t(\tau)$ and $K_s(r)$ given by
\begin{align}
K_{\inh,t}(\tau) &= \frac{1}{3\pi(r_M^3-r_m^3)} \int_W K_{\inh,st}(r,\tau)\dd r \label{eq:checkKt}\\
K_{\inh,s}(r) &=  \frac{1}{\tau_M^2-\tau_m^2} \int_T K_{\inh,st}(r,\tau)\dd \tau.\label{eq:checkKs}
\end{align}
The idea behind~\eqref{eq:checkKt}-\eqref{eq:checkKs} is that under the Poisson assumption $K_{\inh,t}(\tau)=2\tau$ and $K_{\inh,s}(r)=\pi r^2$ which respectively correspond to the Ripley's $K$ function of a one-dimensional and two-dimensional inhomogeneous Poisson point process. Note that $K_{\inh,t}(\tau)$ and $K_{\inh,s}(r)$ do not correspond to the Ripley's $K$ functions of $\Xt$ and $\Xs$. Estimates of $K_{\inh,t}(\tau)= K_{\inh,t}^{\pi_0}(\tau)$ and $K_{\inh,s}(r)=K_{\inh,s}^{\pi_0}(r)$ ensue from the estimate of $K_{\inh,st}(r,\tau)$. As seen from Proposition~\ref{prop:Xstpi}, we use a (homogeneous) subsampled version $\Xst^{\pi_0}$. The estimator, proposed by~\citet{gabriel2013stpp}, is given for some $r,\tau>0$ by
$$
\hat K_{\inh,st}(r,\tau)=\frac{1}{|W\times T|} \sum_{y,y^\prime \in \Xst^{\pi_0}} \frac{e(y,y^\prime)}{\hat \lambda_{s,t}(y)\hat \lambda_{s,t}(y^\prime)} \; \mathbf 1 ( y-y^\prime \in \B).
$$
where $e(\cdot,\cdot)$ is an edge-correction factor (see e.g. \citet{moller2003statistical,gabriel2014estimating}) and where $\hat \lambda_{st}$ is obviously a preliminary estimation of the spatio-temporal intensity (eventually itself estimated using a subsampled version of $\Xst$).

Then, from the estimation of $K_{\inh,t}$ and $K_{\inh,s}$ we construct a combined global envelope test (using the extreme rank length procedure) proposed by~\citet{myllymaki2017global}. To construct envelope, we generate $B=199$ simulations of  inhomogeneous Poisson point processes on $W\times T$. We evaluate 
$\hat K_{\inh,t}(\tau)$ (resp. $\hat K_{\inh,s}(r)$) for 50 values of $\tau$ (resp. $r$) from $\tau_m=0$ (resp. $r_m=0$) to $\tau_{M}=0.75\% \times|T|$ (resp. $r_{M}=20\% \times |W|^{1/2}$). The choice of $r_{M}$ follows from rule of thumb for spatial point processes~\citep{baddeley2015spatial}. The value of $\tau_{M}$ represents a one-month period approximately.  Estimates $\hat K_{\inh,t}$ and $\hat K_{\inh,s}$ and envelopes are represented in Figure~\ref{fig:Kst}. To facilitate the interpretation $x$-axis for Figure~\ref{fig:Kst} (left) we have rescaled time events such that all events occur in time in the interval $(0,1)$. The graphical results together with the global combined p-value equal to $0.05\%$ are unambiguous. The departure to the Poisson modeling is highly significant. In other words, there is strong evidence to reject the inhomogeneous Poisson model for the lightning strikes dataset. The two plots also show that $K_{\inh,t}(\tau)>2\tau$ and $K_{\inh,s}(r) > \pi r^2$ for all values of $r,\tau$ considered and suggest that the $\Xst$ is probably clustered a lot in time and space.
 
\begin{figure}[htbp]
\centering
\includegraphics[width=\textwidth]{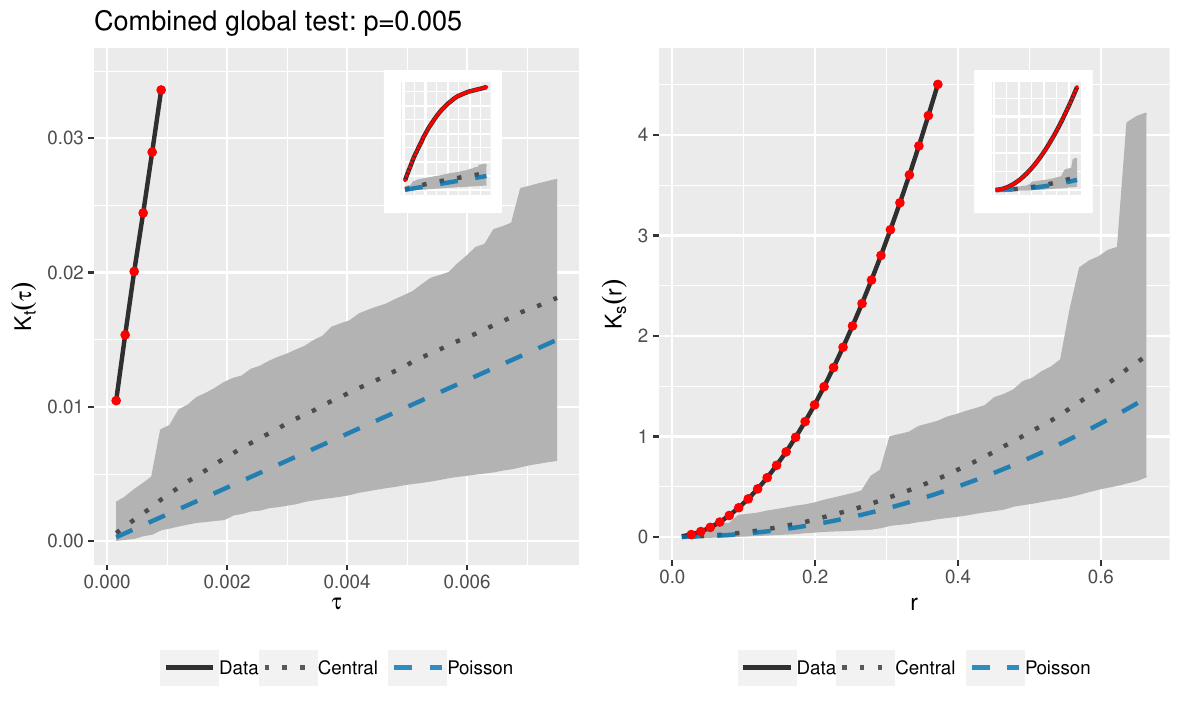}
\caption{Combined global envelope test based on the summary functions $K_{\inh,t}(\tau)$ and $K_{\inh,s}(r)$. We have rescaled time events in $(0,1)$ to have an easier interpretation of the $x$-axis for $K_{\inh,t}(\tau)$ (in particular 0.075 corresponds approximately to a one-month period). Observed curves have been truncated on the $y$-axis for a better visualization. Insets represent the non-truncated versions.}
\label{fig:Kst}
\end{figure}

\section{Conclusion} \label{sec:conclusion}

In this paper, we estimate first and second order structure of the inhomogeneous spatio-temporal point process given by lightning strike impacts. We considered subsampling strategy to tune kernel estimation, to estimate the null distribution of the first-order separability test, to estimate the inhomogeneous K function and its behavior under the Poisson assumption. The results are crystal clear, this point pattern is highly inhomogeneous both in time and space. This dependence is time and space is complex and the process seems to exhibit a high degree of clustering in time and space. These conclusions were addressed using largely reduced subsampled which allows us to reduce significantly computational time and cost. By passing we have shown that a small value of the constant retaining probability  $\pi_0$ prevents from understanding distance-based summary statistics which are almost indistinguishable  from an inhomogeneous Poisson point process. This research leads to many perspectives. From a theoretical point of view, some are presented in~\ref{app:simulation} and concern the strategy (b) to subsample (ie. transform an inhomogeneous point process into a homogeneous one). Others could concern the precision of subsampled estimators like $\hat \lambda^{\pi}_{st}$ or $\hat K_{\inh,st}^\pi$. Propositions~\ref{prop:Xstpi}-\ref{prop:FGJ} have shown that subsampling methodology is a correct procedure but understanding the rate of convergence of nonparametric estimators with the retaining probability function could be of great importance.

From a practical point of view, after the exploratory analysis presented in this paper, the natural next step is to model more specifically the intensity function. We have already presented a few spatial covariates like the altitude map, the binary map induced by the Auer division between the North Alps and South Alps. We also have at our disposal several other spatio-temporal covariates like temperature, levels of pressure, heights of precipitation, etc. These extra information consitute more than 25Go of data and raise many other numerical and practical challenge we expect to investigate in a future research.

\section*{Acknowledgements}

This research is funded by Labex PERSYVAL-lab ANR-11-LABX-0025. The authors would like to take the opportunity to thank first Météorage and Méteo-France and especially Mickaël Zamo for providing us with the data and for scientific exchanges and second Jiri Dvorak for discussions and for providing us with some \texttt{R} code.

\bibliography{refs.bib}

\begin{thebibliography}{}

\bibitem[Auer et~al., 2007]{auer2007histalp}
Auer, I., B{\"o}hm, R., Jurkovic, A., Lipa, W., Orlik, A., Potzmann, R.,
  Sch{\"o}ner, W., Ungersb{\"o}ck, M., Matulla, C., Briffa, K., et~al. (2007).
\newblock Histalp—historical instrumental climatological surface time series
  of the greater alpine region.
\newblock {\em International Journal of Climatology: A Journal of the Royal
  Meteorological Society}, 27(1):17--46.

\bibitem[Baddeley et~al., 2015]{baddeley2015spatial}
Baddeley, A., Rubak, E., and Turner, R. (2015).
\newblock {\em Spatial point patterns: methodology and applications with R}.
\newblock CRC press.

\bibitem[Baddeley et~al., 2005]{baddeley2005residual}
Baddeley, A., Turner, R., M{\o}ller, J., and Hazelton, M. (2005).
\newblock Residual analysis for spatial point processes (with discussion).
\newblock {\em Journal of the Royal Statistical Society Series B: Statistical
  Methodology}, 67(5):617--666.

\bibitem[Barr and Schoenberg, 2010]{barr2010voronoi}
Barr, C.~D. and Schoenberg, F.~P. (2010).
\newblock On the voronoi estimator for the intensity of an inhomogeneous planar
  poisson process.
\newblock {\em Biometrika}, 97(4):977--984.

\bibitem[Bivand et~al., 2008]{bivand2008applied}
Bivand, R.~S., Pebesma, E.~J., G{\'o}mez-Rubio, V., and Pebesma, E.~J. (2008).
\newblock {\em Applied spatial data analysis with R}, volume 747248717.
\newblock Springer.

\bibitem[Blanchet et~al., 2021a]{blanchet-etal-2021b}
Blanchet, J., Blanc, A., and Creutin, J.-D. (2021a).
\newblock Explaining recent trends in extreme precipitation in the southwestern
  alps by changes in atmospheric influences.
\newblock {\em Weather and Climate Extremes}, 33:100356.

\bibitem[Blanchet et~al., 2021b]{blanchet-etal-2021}
Blanchet, J., Creutin, J.-D., and Blanc, A. (2021b).
\newblock Retreating winter and strengthening autumn mediterranean influence on
  extreme precipitation in the southwestern alps over the last 60 years.
\newblock {\em Environmental Research Letters}, 16(034056).

\bibitem[Chiu et~al., 2013]{chiu2013stochastic}
Chiu, S.~N., Stoyan, D., Kendall, W.~S., and Mecke, J. (2013).
\newblock {\em Stochastic geometry and its applications}.
\newblock John Wiley \& Sons.

\bibitem[Cronie et~al., 2023]{cronie2023statistical}
Cronie, O., Moradi, M., and Biscio, C.~A. (2023).
\newblock A cross-validation-based statistical theory for point processes.
\newblock {\em to appear in Biometrika}.

\bibitem[Cronie and Van~Lieshout, 2015]{cronie2015aj}
Cronie, O. and Van~Lieshout, M. N.~M. (2015).
\newblock A $j$-function for inhomogeneous spatio-temporal point processes.
\newblock {\em Scandinavian Journal of Statistics}, 42(2):562--579.

\bibitem[Cronie and Van~Lieshout, 2018]{cronie2018non}
Cronie, O. and Van~Lieshout, M. N.~M. (2018).
\newblock A non-model-based approach to bandwidth selection for kernel
  estimators of spatial intensity functions.
\newblock {\em Biometrika}, 105(2):455--462.

\bibitem[Daley and Vere-Jones, 2008]{daley2008introduction}
Daley, D.~J. and Vere-Jones, D. (2008).
\newblock {\em An Introduction to the Theory of Point Processes. Volume II:
  General Theory and Structure}.
\newblock Springer.

\bibitem[Daley et~al., 2003]{daley2003introduction}
Daley, D.~J., Vere-Jones, D., et~al. (2003).
\newblock {\em An introduction to the theory of point processes: volume I:
  elementary theory and methods}.
\newblock Springer.

\bibitem[Diggle, 2006]{diggle2006spatio}
Diggle, P.~J. (2006).
\newblock Spatio-temporal point processes: methods and applications.
\newblock {\em Monographs on Statistics and Applied Probability}, 107:1.

\bibitem[Diggle, 2013]{diggle2013statistical}
Diggle, P.~J. (2013).
\newblock {\em Statistical analysis of spatial and spatio-temporal point
  patterns}.
\newblock CRC press.

\bibitem[Gabriel, 2014]{gabriel2014estimating}
Gabriel, E. (2014).
\newblock Estimating second-order characteristics of inhomogeneous
  spatio-temporal point processes: influence of edge correction methods and
  intensity estimates.
\newblock {\em Methodology and Computing in Applied Probability}, 16:411--431.

\bibitem[Gabriel et~al., 2013]{gabriel2013stpp}
Gabriel, E., Rowlingson, B.~S., and Diggle, P.~J. (2013).
\newblock stpp: an r package for plotting, simulating and analyzing
  spatio-temporal point patterns.
\newblock {\em Journal of Statistical Software}, 53:1--29.

\bibitem[Ghorbani et~al., 2021]{ghorbani2021testing}
Ghorbani, M., Vafaei, N., Dvo{\v{r}}{\'a}k, J., and Myllym{\"a}ki, M. (2021).
\newblock Testing the first-order separability hypothesis for spatio-temporal
  point patterns.
\newblock {\em Computational Statistics \& Data Analysis}, 161:107245.

\bibitem[Hare et~al., 2018]{hare-etal-2018}
Hare, B.~M., Scholten, O., Bonardi, A., Buitink, S., Corstanje, A., Ebert, U.,
  Falcke, H., Hörandel, J.~R., Leijnse, H., Mitra, P., Mulrey, K., Nelles, A.,
  Rachen, J.~P., Rossetto, L., Rutjes, C., Schellart, P., Thoudam, S., Trinh,
  T. N.~G., ter Veen, S., and Winchen, T. (2018).
\newblock Lofar lightning imaging: Mapping lightning with nanosecond precision.
\newblock {\em Journal of Geophysical Research: Atmospheres},
  123(5):2861--2876.

\bibitem[Mateu et~al., 2023]{mateu2023bayesian}
Mateu, J. et~al. (2023).
\newblock Bayesian approach for modelling spatial--temporal crime data.
\newblock {\em Journal of Statistical Sciences}, 16(2):435--448.

\bibitem[M{\o}ller and Ghorbani, 2012]{moller2012aspects}
M{\o}ller, J. and Ghorbani, M. (2012).
\newblock Aspects of second-order analysis of structured inhomogeneous
  spatio-temporal point processes.
\newblock {\em Statistica Neerlandica}, 66(4):472--491.

\bibitem[M{\o}ller and Waagepetersen, 2003]{moller2003statistical}
M{\o}ller, J. and Waagepetersen, R.~P. (2003).
\newblock {\em Statistical inference and simulation for spatial point
  processes}.
\newblock CRC press.

\bibitem[Moradi et~al., 2019]{moradi2019resample}
Moradi, M.~M., Cronie, O., Rubak, E., Lachieze-Rey, R., Mateu, J., and
  Baddeley, A. (2019).
\newblock Resample-smoothing of voronoi intensity estimators.
\newblock {\em Statistics and computing}, 29(5):995--1010.

\bibitem[Myllym{\"a}ki et~al., 2017]{myllymaki2017global}
Myllym{\"a}ki, M., Mrkvi{\v{c}}ka, T., Grabarnik, P., Seijo, H., and Hahn, U.
  (2017).
\newblock Global envelope tests for spatial processes.
\newblock {\em Journal of the Royal Statistical Society Series B: Statistical
  Methodology}, 79(2):381--404.

\bibitem[Opitz et~al., 2020]{opitz2020point}
Opitz, T., Bonneu, F., and Gabriel, E. (2020).
\newblock Point-process based bayesian modeling of space--time structures of
  forest fire occurrences in mediterranean france.
\newblock {\em Spatial Statistics}, 40:100429.

\bibitem[Raeisi et~al., 2023]{raeisi2023spatio}
Raeisi, M., Bonneu, F., and Gabriel, E. (2023).
\newblock A spatio-temporal hybrid strauss hardcore point process for forest
  fire occurrences.
\newblock {\em arXiv preprint arXiv:2308.06726}.

\bibitem[Serra et~al., 2014]{serra2014spatio}
Serra, L., Saez, M., Mateu, J., Varga, D., Juan, P., D{\'\i}az-{\'A}valos, C.,
  and Rue, H. (2014).
\newblock Spatio-temporal log-gaussian cox processes for modelling wildfire
  occurrence: the case of catalonia, 1994--2008.
\newblock {\em Environmental and ecological statistics}, 21:531--563.

\bibitem[Sheather and Jones, 1991]{sheather1991reliable}
Sheather, S.~J. and Jones, M.~C. (1991).
\newblock A reliable data-based bandwidth selection method for kernel density
  estimation.
\newblock {\em Journal of the Royal Statistical Society: Series B
  (Methodological)}, 53(3):683--690.

\end{thebibliography}

\appendix

\section{Homogeneous point process from a subsampled inhomogeneous point process} \label{app:simulation}

In this section we present a simulation study in order to illustrate the subsampling strategy (b) presented in Section~\ref{sec:subsampling}. This strategy intends to make an inhomogeneous point process a homogeneous one. We use a toy example in the planar case (thus no temporal component). Let $\lambda_s$ be the inhomogeneous spatial intensity of $\Xs$. We propose the following practical algorithm to implement strategy (b). Let $\nu^\pi$, set by the user, be the expected number of points of the subsampled version $\Xs^\pi$. 
\begin{itemize}
\item[(i)] We estimate non parametrically the intensity $\lambda_s$ using the Voronoi-based nonparametric estimator introduced by~\citet{barr2010voronoi} and recently improved by~\citet{moradi2019resample}. This estimator is a histogram-based type estimator based on the Voronoi tessellation generated by the point pattern. This estimator has the interest to estimate quickly $\Delta_\mu=\{x \in W: \lambda_s(x)\ge \mu\}$ by
$$
\hat \Delta_\mu = \left\{
C \in \mathcal V(\Xs): |C|^{-1}\int_C \hat \lambda_s(x)\ge \mu 
\right\}$$
where $\mathcal V(\Xs)$ denotes the Voronoi tessellation in $W$ generated by the points of $\Xs$. Thus $C$ stands for a Voronoi cell and $|C|$ its volume. Note that this procedure is implemented in the function \texttt{densityVoronnoi} in the \texttt{R spatstat} package.
\item[(ii)] We set $\hat \mu= \mathrm{argmin}_\mu \hat{\mathcal{L}}( \mu)$ 
where $\hat{\mathcal{L}}( \mu) = \{\nu^\pi-\mu |\hat \Delta_\mu|\}^2$ is the natural estimator of the theoretical loss function $\mathcal L(\mu)=\{\nu^\pi-\mu | \Delta_\mu|\}^2$.
\item[(iii)] Let $\Xs^{\hat \pi}$ be a subsampled version of $\Xs$ obtained with retention probability $\hat \pi(x) = \frac{\mu}{\hat \lambda_s(x)}$ for any $x\in \hat \Delta_\mu$, thus defined on $\hat \Delta_\mu\subset W$.
\end{itemize}

The output of this algorithm is a realization of $\Xs^{\hat \pi}$ whose intensity is close to $\mu \lambda_s(x)/\hat \lambda_s(x) \approx \mu$ for any $x\in \hat \Delta_\mu$. Thus $\Xs$ is expected to be homogeneous. 
To illustrate this procedure, we consider two models for $\lambda_s$ in the planar case. We focus only on the Poisson case and test the homogeneity assumption of a pattern using the quadrat test procedure \citep{baddeley2015spatial} which is a valid test in this context. Poisson point processes $\Xs$ are generated in $W=[0,1]^2$  with intensity function from the two following models
$$
\lambda_s(x)= \beta \left\{
\begin{array}{cc}
 \exp\left\{
 -\phi^{-1} (1+ |x_1-1/2| \times |x_2-1/2|)
 \right\}
 & \text{ (cross)}\\
 \exp\left\{
 \phi^{-1} \times (|\sin(\pi x_1)+|\sin(\pi x_2)
 \right\}
 & \text{ (sin)} 
\end{array}
\right.
$$
where $x=(x_1,x_2) \in W$, $\phi=0.02$ (resp. $0.2$) for the cross model (resp. sin model) and where we set $\beta$ such that $\nu=\E N(W)= 50,000$ points. Finally we conduct the simulation such that $\nu^\pi=500$ so we keep on average 10\% from the initial pattern. Figure~\ref{fig:model} illustrates the intensity function models.
Figures~\ref{fig:resCross}-\ref{fig:resSin} present empirical results for two simulated patterns $\xs$ from these models (point patterns are illustrated in the first rows of these figures). For each simulation, we adopt the previously described  algorithm. Figure~\ref{fig:loss} shows that the loss functions $\mathcal L(\mu)$ for both models are well estimated. The minimum of the loss function is also well estimated. And it can be checked that for these models and simulations the losses functions equal 0 at these minima. Figures~\ref{fig:resCross}-\ref{fig:resSin} present simulated patterns. First rows correspond to simulations of $\Xs$ on $W=[0,1]^2$ with on average 50,000 points. Second rows are subsampled version $\Xs^\pi$ obtained with the strategy (b) where we used the theoretical retaining probability $\pi$ and thus the true set $\Delta_\mu$. Thus the grey areas correspond to $W\setminus \Delta_\mu$. Finally, third rows correspond to simulations of $\Xs^{\hat \pi}$ obtained with the data-driven procedure described above where we both estimate $\pi$ and $\Delta_\mu$. Visually speaking, we observe that the subsampled versions $\Xs^\pi$ and $\Xs^{\hat\pi}$ tend to produce homogeneous patterns respectively on $\Delta_\mu$ and $\hat \Delta_\mu$. To confirm this, we also perform the quadrat test procedure \citep{baddeley2015spatial} (where the number of tiles follows from the rule of thumb from the \texttt{spatstat} package) for all point patterns (so for $\Xs$, $\Xs^\pi$ and $\Xs^{\hat \pi}$). The title of each graph contains the observed number of points and the p-value of this test (accurate to 0.01\%). Obviously, the rest rejects (with high confidence) the homogeneous Poisson point process (on $W$) assumption for the patterns $\Xs$ (p-values are all smaller than 0.01\%). Values show that the null hypothesis (homogeneous Poisson on $\Delta_\mu$ or $\hat \Delta_\mu$) cannot be rejected at level 5\% for simulations of $\Xs^\pi$ and $\Xs^{\hat \pi}$.

This simulation study suggests that the proposed algorithm which aims to subsample an inhomogeneous point process to make it homogeneous seems possible. It does raise many questions. Among them : (a) are there theoretical guarantees that the estimated loss function has a unique minimum, that the value of the loss function at this minimum is zero? (b) are the findings on these models reproducible for a larger class of models and in particular for spatio-temporal guarantees? (c) can we quantify the gain of having a homogeneous subsample compared to an inhomogeneous one? Definitely, these interesting questions will be the topic of further research.

\begin{figure}[htbp]
\centering
\includegraphics[width=.8\textwidth]{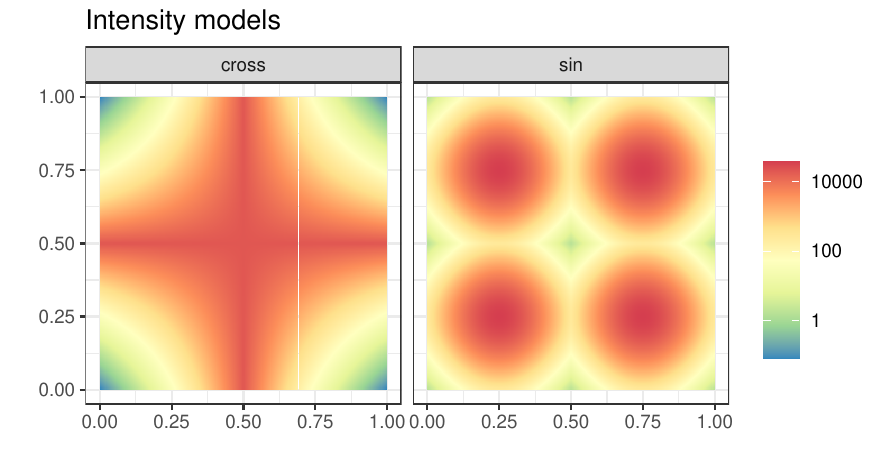}
\caption{Intensity functions of the two models considered in the simulation. The cross model (left) concentrates points around the x-axis or the y-axis, while the sin model tends to produce patterns in four clusters. Colors are in log-scale.}\label{fig:model}
\end{figure}

\begin{figure}[htbp]
\centering
\includegraphics[width=.8\textwidth]{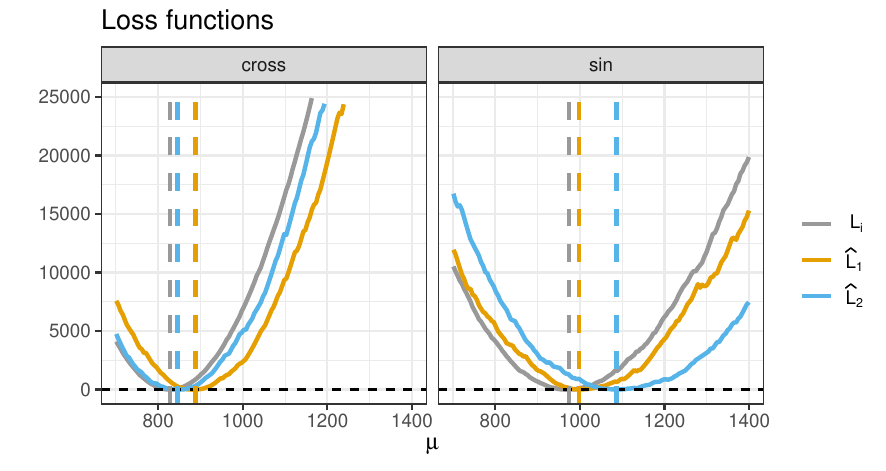}
\caption{Theoretical loss functions $\mathcal L(\mu) $ and estimated loss functions $\hat{\mathcal{L}}_i(\mu)$ for $i=1,2$ the two simulations considered in the simulation study and for the two considered models (cross and sin) - see Figures~\ref{fig:resCross}-\ref{fig:resSin}. Vertical dashed lines correspond to minima of these functions. }\label{fig:loss}
\end{figure}

\begin{figure}[htbp]
\centering
\begin{tabular}{ll}
\includegraphics[width=.43\textwidth]{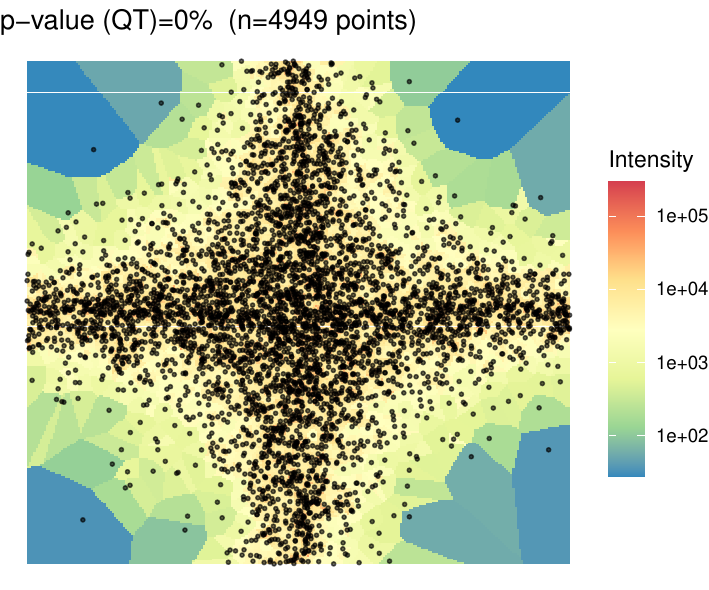}&\includegraphics[width=.43\textwidth]{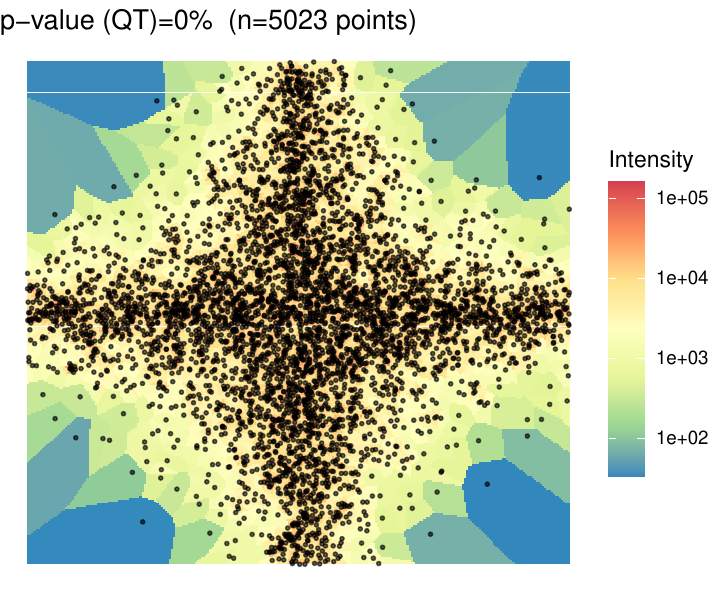}\\
\includegraphics[width=.43\textwidth]{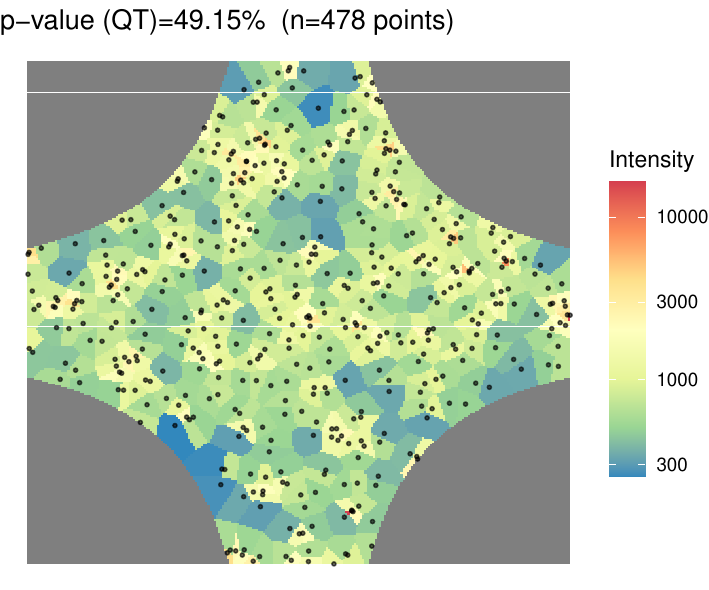}&\includegraphics[width=.43\textwidth]{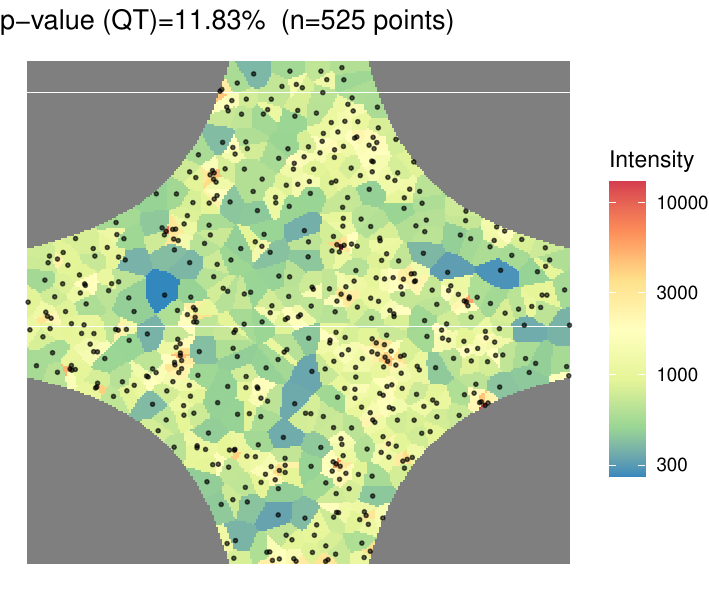}\\
\includegraphics[width=.43\textwidth]{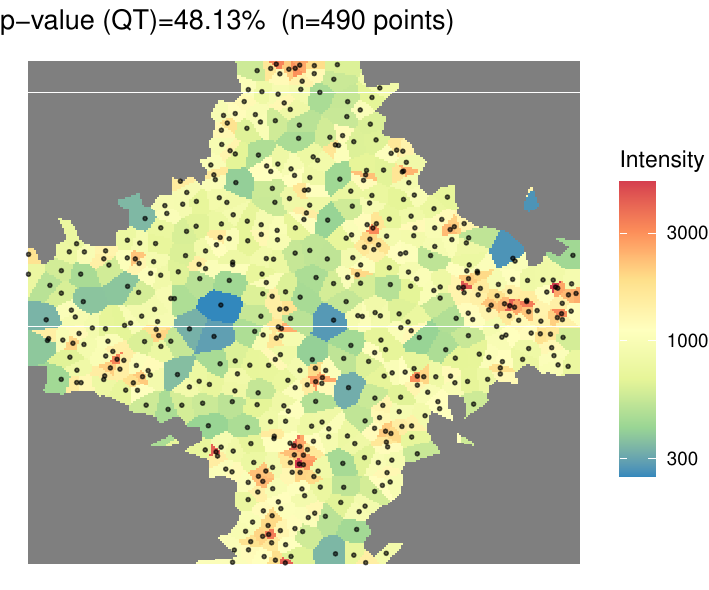}&\includegraphics[width=.43\textwidth]{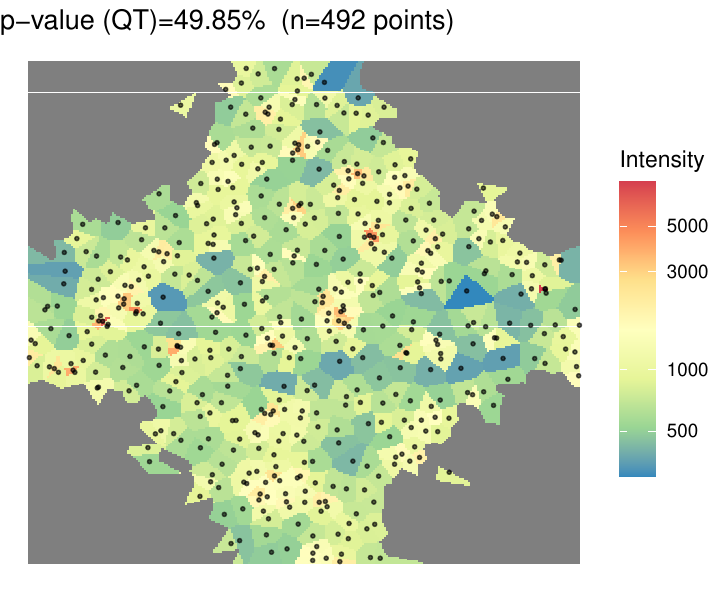}
\end{tabular}
\caption{First row present two simulations with on average 50,000 points of the cross model. Second row (resp. third row) present two subsampled versions $\Xs^\pi$ with the theoretical retaining probability function (resp. $\Xs^{\hat \pi}$). Grey areas are estimates of the set $\Delta_{\hat \mu}$ (second row) and $\hat \Delta_{\hat \mu}$ (third row). Latent image correspond to nonparametric estimations of the Voronoi density estimator on $W$ (first row), $\Delta_{\hat \mu}$ (second row) and $\hat \Delta_{\hat \mu}$ (third row). Titles include the observed number of points and the p-value of the quadrat test (denoted by QT).}
\label{fig:resCross}
\end{figure}

\begin{figure}[htbp]
\centering
\begin{tabular}{ll}
\includegraphics[width=.43\textwidth]{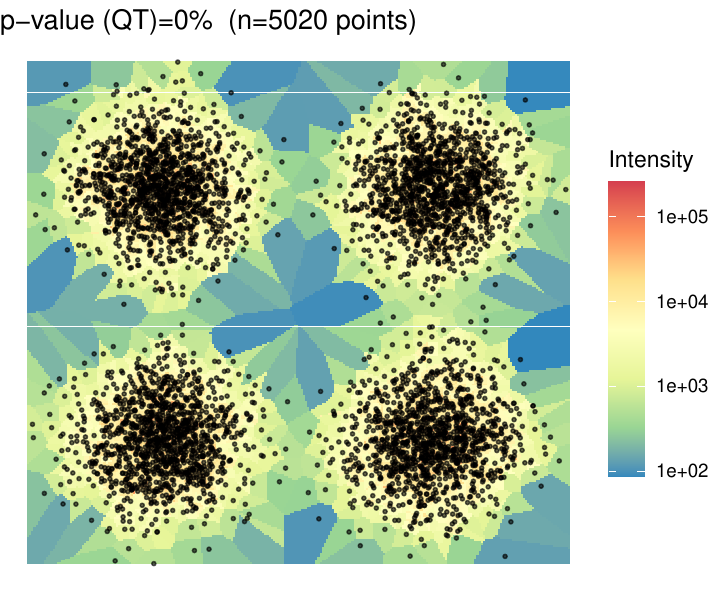}&\includegraphics[width=.43\textwidth]{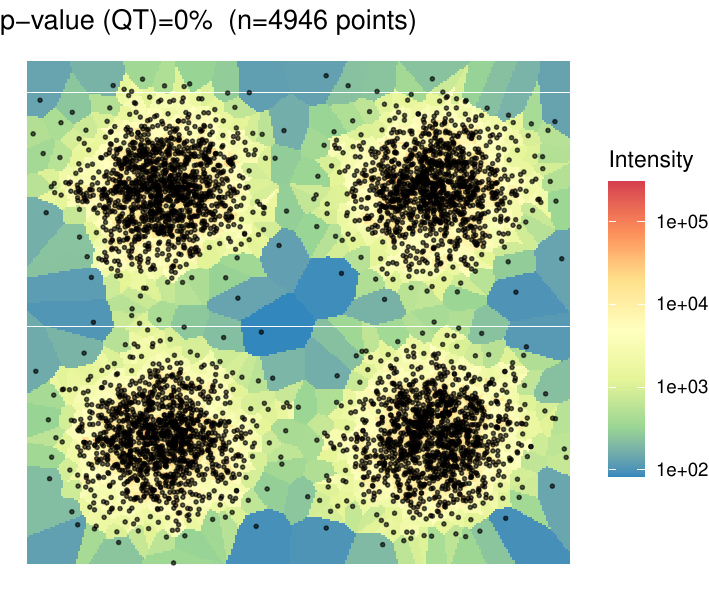}\\
\includegraphics[width=.43\textwidth]{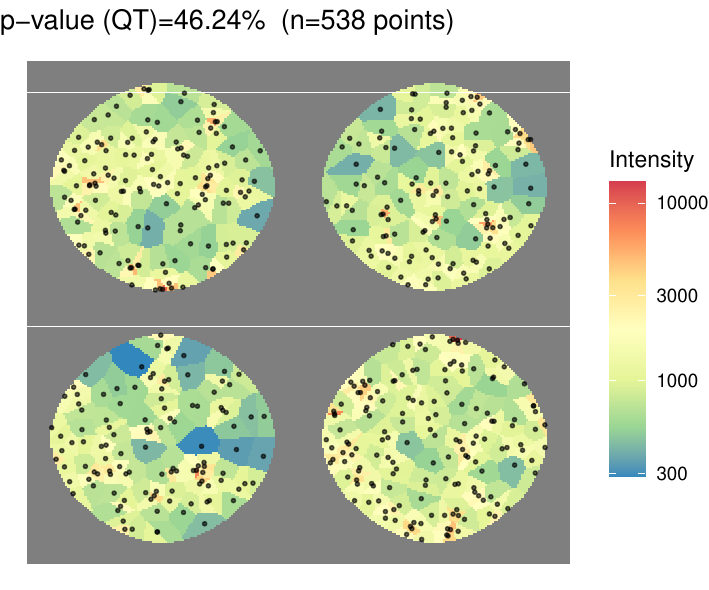}&\includegraphics[width=.43\textwidth]{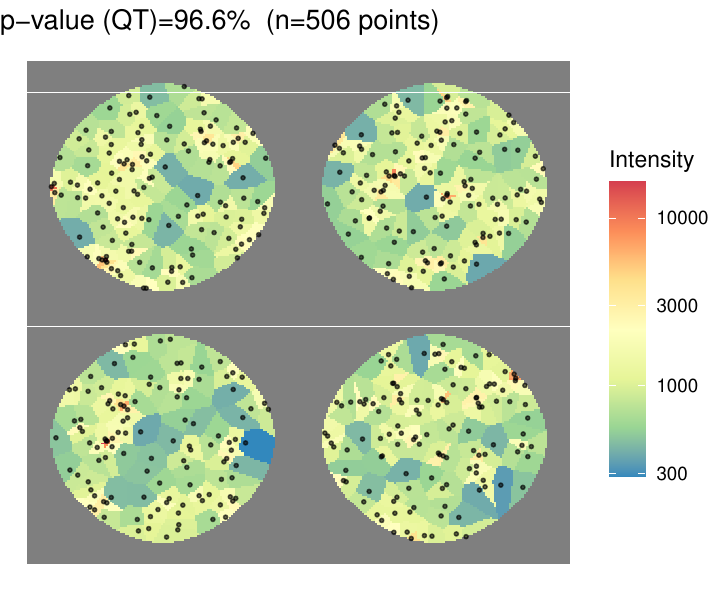}\\
\includegraphics[width=.43\textwidth]{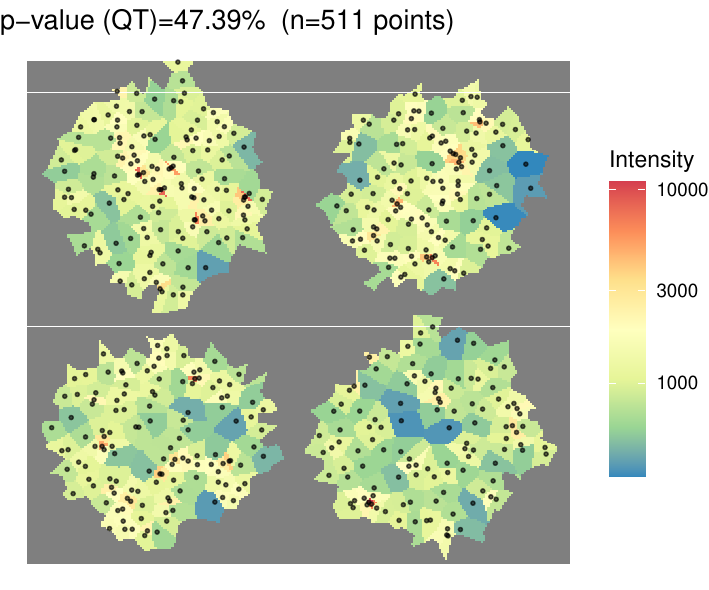}&\includegraphics[width=.43\textwidth]{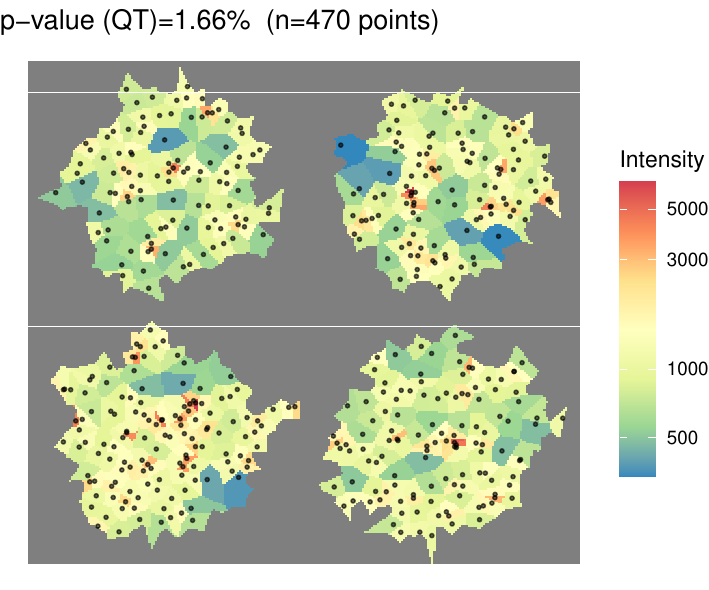}
\end{tabular}
\caption{First row present two simulations with on average 50,000 points of the sin model. Second row (resp. third row) present two subsampled versions $\Xs^\pi$ with the theoretical retaining probability function (resp. $\Xs^{\hat \pi}$). Grey areas are estimates of the set $\Delta_{\hat\mu}$ (second row) and $\hat \Delta_{\hat \mu}$ (third row). Latent image correspond to nonparametric estimations of the Voronoi density estimator on $W$ (first row), $\Delta_{\hat \mu}$ (second row) and $\hat \Delta{\hat \mu}$ (third row). Titles include the observed number of points and the p-value of the quadrat test (denoted by QT).}
\label{fig:resSin}
\end{figure}

\end{document}